\documentclass[a4paper,10pt]{amsart}
\usepackage{enumerate, amsmath, amsfonts, amssymb, amsthm, mathtools, thmtools, wasysym, graphics, graphicx, xcolor, frcursive,xparse,comment,ytableau,stmaryrd,bbm,array,colortbl,tensor}
\usepackage[all]{xy}
\usepackage[boxed,norelsize]{algorithm2e}
\usepackage{hyperref}

\usepackage{url, hypcap}
\hypersetup{colorlinks=true, citecolor=darkblue, linkcolor=darkblue}

\usepackage{tikz}
\usetikzlibrary{calc,through,backgrounds,shapes,matrix}
\definecolor{darkblue}{rgb}{0.0,0,0.7} 
\newcommand{\darkblue}{\color{darkblue}} 
\definecolor{darkred}{rgb}{0.7,0,0} 
\definecolor{lightgrey}{rgb}{0.7,0.7,0.7} 

\definecolor{meet}{RGB}{255,205,111}
\definecolor{join}{RGB}{0,77,178}

\newtheorem{theorem}{Theorem}[section]
\newtheorem{proposition}[theorem]{Proposition}
\newtheorem{corollary}[theorem]{Corollary}
\newtheorem{lemma}[theorem]{Lemma}

\theoremstyle{definition}
\newtheorem{definition}[theorem]{Definition}
\newtheorem{example}[theorem]{Example}

\usepackage[procnames]{listings}
\usepackage[nameinlink]{cleveref}
\usepackage[all]{xy}
\usepackage[T1]{fontenc}
\crefformat{footnote}{#2\footnotemark[#1]#3}
\crefformat{conjecture}{Conjecture~#2#1#3}

\usepackage[cal=boondoxo]{mathalfa}
\usepackage[colorinlistoftodos]{todonotes}

\newcommand{\defn}[1]{\emph{\darkblue #1}}

\newcommand{\LL}{\mathcal{L}}
\newcommand{\JJ}{\mathcal{J}}
\newcommand{\MM}{\mathcal{M}}
\newcommand{\x}{x}
\newcommand{\al}{\mathcal{D}}
\newcommand{\br}{\mathcal{U}}

\newcommand{\y}{y}
\newcommand{\tog}{\mathrm{flip}}
\newcommand{\row}{\mathrm{row}}

\newcommand{\down}{\mathrm{D}}
\newcommand{\up}{\mathrm{U}}

\newcommand{\trip}{\mathrm{top}}
\newcommand{\mut}{\mathrm{tog}}

\usetikzlibrary{math}
\usepackage{xifthen}
\usepackage{xstring}
\SetKwInput{kwset}{set}
\usetikzlibrary{arrows,backgrounds,calc,trees}
\pgfdeclarelayer{background}
\pgfsetlayers{background,main}
\renewcommand{\mod}{\operatorname{mod}}
\newcommand{\Hom}{\operatorname{Hom}}
\newcommand{\End}{\operatorname{End}}
\newcommand{\Ext}{\operatorname{Ext}}
\newcommand{\tors}{\operatorname{tors}}

\newcommand{\convexpath}[2]{
[
    create hullnodes/.code={
        \global\edef\namelist{#1}
        \foreach [count=\counter] \nodename in \namelist {
            \global\edef\numberofnodes{\counter}
            \node at (\nodename) [draw=none,name=hullnode\counter] {};
        }
        \node at (hullnode\numberofnodes) [name=hullnode0,draw=none] {};
        \pgfmathtruncatemacro\lastnumber{\numberofnodes+1}
        \node at (hullnode1) [name=hullnode\lastnumber,draw=none] {};
    },
    create hullnodes
]
($(hullnode1)!#2!-90:(hullnode0)$)
\foreach [
    evaluate=\currentnode as \previousnode using \currentnode-1,
    evaluate=\currentnode as \nextnode using \currentnode+1
    ] \currentnode in {1,...,\numberofnodes} {
  let
    \p1 = ($(hullnode\currentnode)!#2!-90:(hullnode\previousnode)$),
    \p2 = ($(hullnode\currentnode)!#2!90:(hullnode\nextnode)$),
    \p3 = ($(\p1) - (hullnode\currentnode)$),
    \n1 = {atan2(\y3,\x3)},
    \p4 = ($(\p2) - (hullnode\currentnode)$),
    \n2 = {atan2(\y4,\x4)},
    \n{delta} = {-Mod(\n1-\n2,360)}
  in
    {-- (\p1) arc[start angle=\n1, delta angle=\n{delta}, radius=#2] -- (\p2)}
}
-- cycle
}

\title{Independence Posets}

\author[H.~Thomas]{Hugh Thomas}
\address[H.~Thomas]{LaCIM, Universit\'e du Qu\'ebec \`a Montr\'eal}
\email{hugh.ross.thomas@gmail.com}

\author[N.~Williams]{Nathan Williams}
\address[N.~Williams]{University of Texas at Dallas}
\email{nathan.f.williams@gmail.com}

\date{\today}

\begin{document}

\maketitle

\begin{abstract}
Let $G$ be an acyclic directed graph. For each vertex $g \in G$, we define an involution on the independent sets of $G$.  We call these involutions flips, and use them to define a new partial order on independent sets of $G$.

Trim lattices generalize distributive lattices by removing the graded hypothesis: a graded trim lattice is a distributive lattice, and every distributive lattice is trim.  Our independence posets are a further generalization of distributive lattices, eliminating also the lattice requirement: an independence poset that is a lattice is always a trim lattice, and every trim lattice is the independence poset for a unique (up to isomorphism) acyclic directed graph $G$. We characterize when an independence poset is a lattice with a graph-theoretic condition on $G$.

We generalize the definition of rowmotion from distributive lattices to independence posets, and we show it can be computed in three different ways.  We also relate our constructions to torsion classes, semibricks, and 2-simpleminded collections arising in the representation theory of certain finite-dimensional directed algebras.
\end{abstract}

\section{Introduction}

In this paper, we always take $G$ to be a finite acylic directed graph. The transitive closure of $G$ defines a poset, which we refer to as \defn{$G$-order}.  Our convention is that $g_1\geq g_2$ in $G$-order if and only if there is a directed path in $G$ from $g_1$ to $g_2$; when we compare vertices of $G$, we will always mean a comparison in $G$-order.  We write $\simeq$ for an isomorphism of posets.

\subsection{Independent sets and tight orthogonal pairs}
Recall that an \defn{independent set} $\mathcal{A} \subseteq G$ is a set of pairwise non-adjacent vertices of $G$.  As we now explain, the orientation provided by $G$ allows us to complete an independent set to a pair of independent sets, either of which determines the other.

\begin{definition}
A pair $(\al,\br)$ of independent sets of $G$ is called \defn{orthogonal} if there is no edge in $G$ from an element of $\al$ to an element of $\br$.  An orthogonal pair of independent sets $(\al,\br)$ is called \defn{tight} if whenever any element of $\al$ is increased (removed and replaced by a larger element with respect to $G$-order) or any element of $\br$ is decreased, or a new element is added to either $\al$ or $\br$, then the result is no longer an orthogonal pair of independent sets.  We abbreviate {\bf t}ight {\bf o}rthogonal {\bf p}air by \defn{top}, and we write \defn{$\trip(G)$} for the set of all tops of $G$.
\label{def:top}
\end{definition}

Some examples are given in~\Cref{fig:first_examples}.

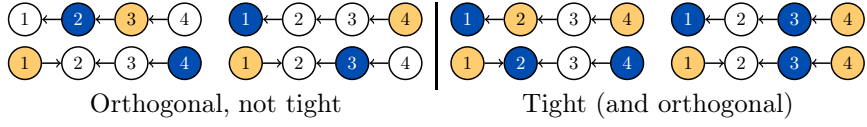
\begin{figure}[htbp]
\[\begin{array}{cc|cc}
\scalebox{.7}{\begin{tikzpicture}
\node (1) [circle,thick,draw,fill=white] at (0,0) {1};
\node (2) [circle,thick,draw,fill=join,text=white] at (1,0) {2};
\node (3) [circle,thick,draw,fill=meet] at (2,0) {3};
\node (4) [circle,thick,draw,fill=white] at (3,0) {4};
\draw[->,thick] (2) to (1);
\draw[->,thick] (3) to (2);
\draw[->,thick] (4) to (3);
\end{tikzpicture}}&
\scalebox{.7}{\begin{tikzpicture}
\node (1) [circle,thick,draw,fill=join,text=white] at (0,0) {1};
\node (2) [circle,thick,draw,fill=white] at (1,0) {2};
\node (3) [circle,thick,draw,fill=white] at (2,0) {3};
\node (4) [circle,thick,draw,fill=meet] at (3,0) {4};
\draw[->,thick] (2) to (1);
\draw[->,thick] (3) to (2);
\draw[->,thick] (4) to (3);
\end{tikzpicture}}&
\scalebox{.7}{\begin{tikzpicture}
\node (1) [circle,thick,draw,fill=join,text=white] at (0,0) {1};
\node (2) [circle,thick,draw,fill=meet] at (1,0) {2};
\node (3) [circle,thick,draw,fill=white] at (2,0) {3};
\node (4) [circle,thick,draw,fill=meet] at (3,0) {4};
\draw[->,thick] (2) to (1);
\draw[->,thick] (3) to (2);
\draw[->,thick] (4) to (3);
\end{tikzpicture}}&
\scalebox{.7}{\begin{tikzpicture}
\node (1) [circle,thick,draw,fill=join,text=white] at (0,0) {1};
\node (2) [circle,thick,draw,fill=white] at (1,0) {2};
\node (3) [circle,thick,draw,fill=join,text=white] at (2,0) {3};
\node (4) [circle,thick,draw,fill=meet] at (3,0) {4};
\draw[->,thick] (2) to (1);
\draw[->,thick] (3) to (2);
\draw[->,thick] (4) to (3);
\end{tikzpicture}}\\
\scalebox{0.7}{\begin{tikzpicture}
\node (1) [circle,thick,draw,fill=meet] at (0,0) {1};
\node (2) [circle,thick,draw,fill=white] at (1,0) {2};
\node (3) [circle,thick,draw,fill=white] at (2,0) {3};
\node (4) [circle,thick,draw,fill=join,text=white] at (3,0) {4};
\draw[->,thick] (1) to (2);
\draw[->,thick] (3) to (2);
\draw[->,thick] (4) to (3);
\end{tikzpicture}}&
\scalebox{0.7}{\begin{tikzpicture}
\node (1) [circle,thick,draw,fill=meet] at (0,0) {1};
\node (2) [circle,thick,draw,fill=white] at (1,0) {2};
\node (3) [circle,thick,draw,fill=join,text=white] at (2,0) {3};
\node (4) [circle,thick,draw,fill=white] at (3,0) {4};
\draw[->,thick] (1) to (2);
\draw[->,thick] (3) to (2);
\draw[->,thick] (4) to (3);
\end{tikzpicture}}&
\scalebox{0.7}{\begin{tikzpicture}
\node (1) [circle,thick,draw,fill=meet] at (0,0) {1};
\node (2) [circle,thick,draw,fill=join,text=white] at (1,0) {2};
\node (3) [circle,thick,draw,fill=white] at (2,0) {3};
\node (4) [circle,thick,draw,fill=join,text=white] at (3,0) {4};
\draw[->,thick] (1) to (2);
\draw[->,thick] (3) to (2);
\draw[->,thick] (4) to (3);
\end{tikzpicture}}&
\scalebox{0.7}{\begin{tikzpicture}
\node (1) [circle,thick,draw,fill=meet] at (0,0) {1};
\node (2) [circle,thick,draw,fill=white] at (1,0) {2};
\node (3) [circle,thick,draw,fill=join,text=white] at (2,0) {3};
\node (4) [circle,thick,draw,fill=meet] at (3,0) {4};
\draw[->,thick] (1) to (2);
\draw[->,thick] (3) to (2);
\draw[->,thick] (4) to (3);
\end{tikzpicture}}\\
\multicolumn{2}{c}{\text{Orthogonal, not tight}} & \multicolumn{2}{c}{\text{Tight (and orthogonal)}}
\end{array}\]
\caption{Eight pairs of independent sets $(\al,\br)$ for two different orientations of a path graph.  The blue vertices correspond to the elements of $\al$, while the orange vertices correspond to $\br$. }
\label{fig:first_examples}
\end{figure}

An independent set can be completed to a tight orthogonal pair in exactly two ways (see~\Cref{map:oneG,map:zeroG}, and~\Cref{fig:alg_ex23}).

\begin{theorem}
\label{thm:trip_ind}
Let $\mathcal{I}$ be an independent set of a directed acyclic graph $G$.  Then there exists a unique $(\mathcal{I},\br) \in \trip(G)$ and a unique $(\al,\mathcal{I})\in \trip(G)$.
\end{theorem}

\subsection{Flips and the independence poset}

Fix $\ell$ a linear extension of $G$-order and $\ell'$ a reverse linear extension of $G$-order.  Note that by our conventions, a linear extension is a linear order such that if there is an edge $g_1\to g_2$, then $g_2$ precedes $g_1$ in the linear extension.

\begin{definition}
The \defn{flip} of $(\al,\br) \in \trip(G)$ at an element $g \in G$ is the tight orthogonal pair $\tog_g(\al,\br)$ defined as follows (see~\Cref{fig:a_toggle} for an example): if $g \not \in \al$ and $g \not \in \br$, the flip does nothing.  Otherwise, preserve all elements of $\al$ that are not less than $g$ and all elements of $\br$ that are not greater than $g$ (and delete all other elements); after switching the set to which $g$ belongs, then greedily complete $\al$ and $\br$ to a tight orthogonal pair in the orders $\ell'$ and $\ell$, respectively.
\label{def:flip}
\end{definition}

Pseudocode for \Cref{def:flip} is given in~\Cref{map:flip}.   \Cref{prop:flips_work} proves that the algorithm produces a tight orthogonal pair, while \Cref{lem:flips_invs} proves that flips are involutions.

\begin{figure}[htbp]
\raisebox{-0.5\height}{\begin{tikzpicture}[scale=.7]
	\tikzmath{\n=6;\m=6;};
    \node at (3.3,3.2) {$g$};
    \node[anchor=west] at (0,6.6) {$\textcolor{join}{\bullet}$ stays};
    \node[anchor=east] at (6,-.6) {$\textcolor{meet}{\bullet}$ stays};
    \node[anchor=east] at (6,6.6) {$\textcolor{meet}{\bullet},\textcolor{join}{\bullet}$ stay};
     \node[anchor=west] at (0,-.6) {$\textcolor{meet}{\bullet},\textcolor{join}{\bullet}$ stay};
 	\foreach \i in {0,...,\n}{
	  \foreach \j in {0,...,\m}{
          \draw[shape=circle,black,fill=white] (\i,\j) circle (.5ex);
      };
	};
    \foreach \i in {0,...,5}{
	  \foreach \j in {0,...,\m}{
          \draw[->] (\i+.1,\j) to (\i+.9,\j);
      };
	};
    \foreach \i in {0,...,6}{
	  \foreach \j in {1,...,\m}{
           \draw[->] (\i,\j-.1) to (\i,\j-.9);
      };
	};
    \foreach \i in {(0,6),(0,4),(0,1),(1,2),(1,0),(2,4),(2,1),(3,5),(3,3),(3,0),(5,4),(5,2),(5,0),(6,6),(6,1)}{
          \draw[shape=circle,fill=meet] \i circle (.8ex);
	};
     \foreach \i in {(0,3),(1,6),(3,2),(4,6),(4,3),(4,1),(6,5),(6,3),(6,0)}{
          \draw[shape=circle,fill=join] \i circle (.8ex);
	};
    \draw[dotted] (-1,2.5) to (3.5,2.5) to (3.5,7);
    \draw[dotted] (7,3.5) to (2.5,3.5) to (2.5,-1);
  \end{tikzpicture}}
$\xleftrightarrow{\tog_g}$
\raisebox{-0.5\height}{
\begin{tikzpicture}[scale=.7]
	\tikzmath{\n=6;\m=6;};
    \node at (3.3,3.2) {$g$};
    \node[anchor=west] at (0,6.6) {$\textcolor{join}{\bullet}$ stays};
    \node[anchor=east] at (6,-.6) {$\textcolor{meet}{\bullet}$ stays};
    \node[anchor=east] at (6,6.6) {$\textcolor{meet}{\bullet},\textcolor{join}{\bullet}$ stay};
     \node[anchor=west] at (0,-.6) {$\textcolor{meet}{\bullet},\textcolor{join}{\bullet}$ stay};
 	\foreach \i in {0,...,\n}{
	  \foreach \j in {0,...,\m}{
          \draw[shape=circle,black,fill=white] (\i,\j) circle (.5ex);
      };
	};
    \foreach \i in {0,...,5}{
	  \foreach \j in {0,...,\m}{
          \draw[->] (\i+.1,\j) to (\i+.9,\j);
      };
	};
    \foreach \i in {0,...,6}{
	  \foreach \j in {1,...,\m}{
           \draw[->] (\i,\j-.1) to (\i,\j-.9);
      };
	};
    \foreach \i in {(0,5),(0,1),(1,4),(1,2),(1,0),(2,5),(2,3),(2,1),(3,6),(3,4),(3,0),(5,4),(5,2),(5,0),(6,6),(6,1)}{
          \draw[shape=circle,fill=meet] \i circle (.8ex);
	};
     \foreach \i in {(0,3),(1,6),(3,3),(4,6),(4,1),(6,5),(6,3),(6,0)}{
          \draw[shape=circle,fill=join] \i circle (.8ex);
	};
    \draw[dotted] (-1,2.5) to (3.5,2.5) to (3.5,7);
    \draw[dotted] (7,3.5) to (2.5,3.5) to (2.5,-1);
  \end{tikzpicture}}
\caption{A flip on a top $(\al,\br)$ in the $7\times 7$ grid oriented from top left to bottom right.  As in~\Cref{fig:first_examples}, the blue vertices correspond to the elements of $\al$, while the orange vertices correspond to the elements of $\br$.  Flipping at the vertex $g$ changes its color, and divides the grid into 5 connected regions (delineated by the dotted lines): the blue vertices not less than $g$ (i.e., not in the bottom right) and the orange vertices not greater than $g$ (i.e., not in the top left) are preserved by the flip.  The orange vertices in the top left are filled in greedily from bottom right to top left; the blue vertices in the bottom right are filled in greedily from top left to bottom right.}
\label{fig:a_toggle}
\end{figure}
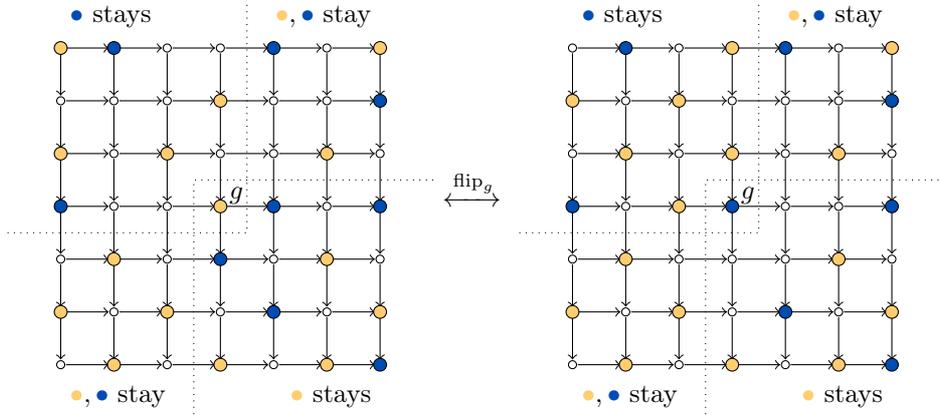

\begin{definition}
\label{def:independence_poset}
We define the \defn{independence poset} on $\trip(G)$ as the reflexive and transitive closure of the relations $(\al,\br) \lessdot (\al',\br')$ if there is some $g \in \br$ such that $\tog_g(\al,\br)=(\al',\br')$.  (\Cref{lem:flip_eq_cov} proves that this really does define a poset, and that these relations are exactly its covers.)
\end{definition}

We denote this poset by $\trip(G)$.  By construction, $\trip(G)$ is connected and has a minimum and a maximum element.  \Cref{fig:examples} gives some examples of independence posets on various orientations of a path of length four, while~\Cref{fig:tamari} realizes the Tamari lattice on 14 elements as an independence poset.  As we summarize below, trim lattices are special cases of independence posets, and so the class of independence posets includes all distributive lattices, Tamari lattices, Cambrian lattices, Fuss-Cambrian lattices, and torsion pairs of tilted finite type hereditary Artin algebras.

\begin{theorem}
\label{thm:tree_structure}
Fix a linear extension $\ell$ of $G$-order.  Flipping only in increasing order of $\ell$ gives a tree structure on the independent sets of $G$.
\end{theorem}

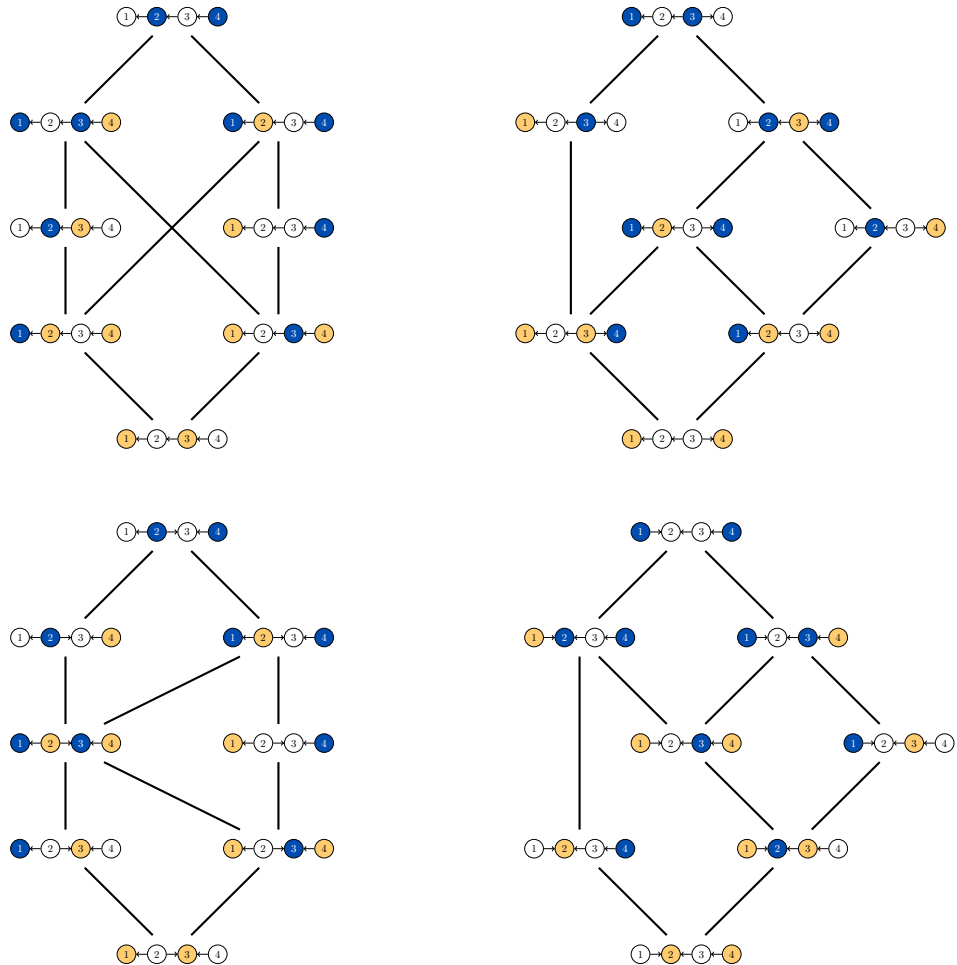
\begin{figure}[htbp]
\raisebox{-0.5\height}{\begin{tikzpicture}[scale=1.4]
\node (a) [align=center] at (0,0) {\scalebox{0.4}{\begin{tikzpicture}
\node (1) [circle,thick,draw,fill=meet] at (0,0) {1};
\node (2) [circle,thick,draw,fill=white] at (1,0) {2};
\node (3) [circle,thick,draw,fill=meet] at (2,0) {3};
\node (4) [circle,thick,draw,fill=white] at (3,0) {4};
\draw[->,thick] (2) to (1);
\draw[->,thick] (3) to (2);
\draw[->,thick] (4) to (3);
\end{tikzpicture}}};
\node (b) [align=center] at (-1,1) {\scalebox{0.4}{\begin{tikzpicture}
\node (1) [circle,thick,draw,fill=join,text=white] at (0,0) {1};
\node (2) [circle,thick,draw,fill=meet] at (1,0) {2};
\node (3) [circle,thick,draw,fill=white] at (2,0) {3};
\node (4) [circle,thick,draw,fill=meet] at (3,0) {4};
\draw[->,thick] (2) to (1);
\draw[->,thick] (3) to (2);
\draw[->,thick] (4) to (3);
\end{tikzpicture}}};
\node (c) [align=center] at (1,2) {\scalebox{0.4}{\begin{tikzpicture}
\node (1) [circle,thick,draw,fill=meet] at (0,0) {1};
\node (2) [circle,thick,draw,fill=white] at (1,0) {2};
\node (3) [circle,thick,draw,fill=white] at (2,0) {3};
\node (4) [circle,thick,draw,fill=join,text=white] at (3,0) {4};
\draw[->,thick] (2) to (1);
\draw[->,thick] (3) to (2);
\draw[->,thick] (4) to (3);
\end{tikzpicture}}};
\node (f) [align=center] at (-1,2) {\scalebox{0.4}{\begin{tikzpicture}
\node (1) [circle,thick,draw,fill=white] at (0,0) {1};
\node (2) [circle,thick,draw,fill=join,text=white] at (1,0) {2};
\node (3) [circle,thick,draw,fill=meet] at (2,0) {3};
\node (4) [circle,thick,draw,fill=white] at (3,0) {4};
\draw[->,thick] (2) to (1);
\draw[->,thick] (3) to (2);
\draw[->,thick] (4) to (3);
\end{tikzpicture}}};
\node (d) [align=center] at (1,1) {\scalebox{0.4}{\begin{tikzpicture}
\node (1) [circle,thick,draw,fill=meet] at (0,0) {1};
\node (2) [circle,thick,draw,fill=white] at (1,0) {2};
\node (3) [circle,thick,draw,fill=join,text=white] at (2,0) {3};
\node (4) [circle,thick,draw,fill=meet] at (3,0) {4};
\draw[->,thick] (2) to (1);
\draw[->,thick] (3) to (2);
\draw[->,thick] (4) to (3);
\end{tikzpicture}}};
\node (e) [align=center] at (1,3) {\scalebox{0.4}{\begin{tikzpicture}
\node (1) [circle,thick,draw,fill=join,text=white] at (0,0) {1};
\node (2) [circle,thick,draw,fill=meet] at (1,0) {2};
\node (3) [circle,thick,draw,fill=white] at (2,0) {3};
\node (4) [circle,thick,draw,fill=join,text=white] at (3,0) {4};
\draw[->,thick] (2) to (1);
\draw[->,thick] (3) to (2);
\draw[->,thick] (4) to (3);
\end{tikzpicture}}};
\node (g) [align=center] at (-1,3) {\scalebox{0.4}{\begin{tikzpicture}
\node (1) [circle,thick,draw,fill=join,text=white] at (0,0) {1};
\node (2) [circle,thick,draw,fill=white] at (1,0) {2};
\node (3) [circle,thick,draw,fill=join,text=white] at (2,0) {3};
\node (4) [circle,thick,draw,fill=meet] at (3,0) {4};
\draw[->,thick] (2) to (1);
\draw[->,thick] (3) to (2);
\draw[->,thick] (4) to (3);
\end{tikzpicture}}};
\node (i) [align=center] at (0,4) {\scalebox{0.4}{\begin{tikzpicture}
\node (1) [circle,thick,draw,fill=white] at (0,0) {1};
\node (2) [circle,thick,draw,fill=join,text=white] at (1,0) {2};
\node (3) [circle,thick,draw,fill=white] at (2,0) {3};
\node (4) [circle,thick,draw,fill=join,text=white] at (3,0) {4};
\draw[->,thick] (2) to (1);
\draw[->,thick] (3) to (2);
\draw[->,thick] (4) to (3);
\end{tikzpicture}}};
\draw[-,thick] (a) to  (b) to  (e) to  (i);
\draw[-,thick] (b) to (f) to (g);
\draw[-,thick] (d) to (c) to (e);
\draw[-,thick] (g) to (i);
\draw[-,thick] (a) to (d);
\draw[-,thick] (d) to  (g);
\end{tikzpicture}} \hfill
\raisebox{-0.5\height}{\begin{tikzpicture}[scale=1.4]
\node (a) [align=center] at (0,0) {\scalebox{0.4}{\begin{tikzpicture}
\node (1) [circle,thick,draw,fill=meet] at (0,0) {1};
\node (2) [circle,thick,draw,fill=white] at (1,0) {2};
\node (3) [circle,thick,draw,fill=white] at (2,0) {3};
\node (4) [circle,thick,draw,fill=meet] at (3,0) {4};
\draw[->,thick]  (2) to (1);
\draw[->,thick] (3) to (2);
\draw[->,thick] (3) to (4);
\end{tikzpicture}}};
\node (b) [align=center] at (-1,1) {\scalebox{0.4}{\begin{tikzpicture}
\node (1) [circle,thick,draw,fill=meet] at (0,0) {1};
\node (2) [circle,thick,draw,fill=white] at (1,0) {2};
\node (3) [circle,thick,draw,fill=meet] at (2,0) {3};
\node (4) [circle,thick,draw,fill=join,text=white] at (3,0) {4};
\draw[->,thick]  (2) to (1);
\draw[->,thick] (3) to (2);
\draw[->,thick] (3) to (4);
\end{tikzpicture}}};
\node (c) [align=center] at (2,2) {\scalebox{0.4}{\begin{tikzpicture}
\node (1) [circle,thick,draw,fill=white] at (0,0) {1};
\node (2) [circle,thick,draw,fill=join,text=white] at (1,0) {2};
\node (3) [circle,thick,draw,fill=white] at (2,0) {3};
\node (4) [circle,thick,draw,fill=meet] at (3,0) {4};
\draw[->,thick]  (2) to (1);
\draw[->,thick] (3) to (2);
\draw[->,thick] (3) to (4);
\end{tikzpicture}}};
\node (f) [align=center] at (0,2) {\scalebox{0.4}{\begin{tikzpicture}
\node (1) [circle,thick,draw,fill=join,text=white] at (0,0) {1};
\node (2) [circle,thick,draw,fill=meet] at (1,0) {2};
\node (3) [circle,thick,draw,fill=white] at (2,0) {3};
\node (4) [circle,thick,draw,fill=join,text=white] at (3,0) {4};
\draw[->,thick]  (2) to (1);
\draw[->,thick] (3) to (2);
\draw[->,thick] (3) to (4);
\end{tikzpicture}}};
\node (d) [align=center] at (1,1) {\scalebox{0.4}{\begin{tikzpicture}
\node (1) [circle,thick,draw,fill=join,text=white] at (0,0) {1};
\node (2) [circle,thick,draw,fill=meet] at (1,0) {2};
\node (3) [circle,thick,draw,fill=white] at (2,0) {3};
\node (4) [circle,thick,draw,fill=meet] at (3,0) {4};
\draw[->,thick]  (2) to (1);
\draw[->,thick] (3) to (2);
\draw[->,thick] (3) to (4);
\end{tikzpicture}}};
\node (e) [align=center] at (1,3) {\scalebox{0.4}{\begin{tikzpicture}
\node (1) [circle,thick,draw,fill=white] at (0,0) {1};
\node (2) [circle,thick,draw,fill=join,text=white] at (1,0) {2};
\node (3) [circle,thick,draw,fill=meet] at (2,0) {3};
\node (4) [circle,thick,draw,fill=join,text=white] at (3,0) {4};
\draw[->,thick]  (2) to (1);
\draw[->,thick] (3) to (2);
\draw[->,thick] (3) to (4);
\end{tikzpicture}}};
\node (g) [align=center] at (-1,3) {\scalebox{0.4}{\begin{tikzpicture}
\node (1) [circle,thick,draw,fill=meet] at (0,0) {1};
\node (2) [circle,thick,draw,fill=white] at (1,0) {2};
\node (3) [circle,thick,draw,fill=join,text=white] at (2,0) {3};
\node (4) [circle,thick,draw,fill=white] at (3,0) {4};
\draw[->,thick]  (2) to (1);
\draw[->,thick] (3) to (2);
\draw[->,thick] (3) to (4);
\end{tikzpicture}}};
\node (i) [align=center] at (0,4) {\scalebox{0.4}{\begin{tikzpicture}
\node (1) [circle,thick,draw,fill=join,text=white] at (0,0) {1};
\node (2) [circle,thick,draw,fill=white] at (1,0) {2};
\node (3) [circle,thick,draw,fill=join,text=white] at (2,0) {3};
\node (4) [circle,thick,draw,fill=white] at (3,0) {4};
\draw[->,thick] (2) to (1);
\draw[->,thick] (3) to (2);
\draw[->,thick] (3) to (4);
\end{tikzpicture}}};
\draw[-,thick] (a) to  (b);
\draw[-,thick] (f) to (e) to  (i);
\draw[-,thick] (b) to (g);
\draw[-,thick] (b) to (f);
\draw[-,thick] (d) to (f);
\draw[-,thick] (d) to (c) to (e);
\draw[-,thick] (g) to (i);
\draw[-,thick] (a) to (d);
\end{tikzpicture}}
\vspace{2em}

\raisebox{-0.5\height}{\begin{tikzpicture}[scale=1.4]
\node (a) [align=center] at (0,0) {\scalebox{0.4}{\begin{tikzpicture}
\node (1) [circle,thick,draw,fill=meet] at (0,0) {1};
\node (2) [circle,thick,draw,fill=white] at (1,0) {2};
\node (3) [circle,thick,draw,fill=meet] at (2,0) {3};
\node (4) [circle,thick,draw,fill=white] at (3,0) {4};
\draw[->,thick] (2) to (1);
\draw[->,thick] (2) to (3);
\draw[->,thick] (4) to (3);
\end{tikzpicture}}};
\node (b) [align=center] at (-1,1) {\scalebox{0.4}{\begin{tikzpicture}
\node (1) [circle,thick,draw,fill=join,text=white] at (0,0) {1};
\node (2) [circle,thick,draw,fill=white] at (1,0) {2};
\node (3) [circle,thick,draw,fill=meet] at (2,0) {3};
\node (4) [circle,thick,draw,fill=white] at (3,0) {4};
\draw[->,thick] (2) to (1);
\draw[->,thick] (2) to (3);
\draw[->,thick] (4) to (3);
\end{tikzpicture}}};
\node (c) [align=center] at (1,2) {\scalebox{0.4}{\begin{tikzpicture}
\node (1) [circle,thick,draw,fill=meet] at (0,0) {1};
\node (2) [circle,thick,draw,fill=white] at (1,0) {2};
\node (3) [circle,thick,draw,fill=white] at (2,0) {3};
\node (4) [circle,thick,draw,fill=join,text=white] at (3,0) {4};
\draw[->,thick] (2) to (1);
\draw[->,thick](2) to (3);
\draw[->,thick] (4) to (3);
\end{tikzpicture}}};
\node (f) [align=center] at (-1,2) {\scalebox{0.4}{\begin{tikzpicture}
\node (1) [circle,thick,draw,fill=join,text=white] at (0,0) {1};
\node (2) [circle,thick,draw,fill=meet] at (1,0) {2};
\node (3) [circle,thick,draw,fill=join,text=white] at (2,0) {3};
\node (4) [circle,thick,draw,fill=meet] at (3,0) {4};
\draw[->,thick] (2) to (1);
\draw[->,thick] (2) to (3);
\draw[->,thick] (4) to (3);
\end{tikzpicture}}};
\node (d) [align=center] at (1,1) {\scalebox{0.4}{\begin{tikzpicture}
\node (1) [circle,thick,draw,fill=meet] at (0,0) {1};
\node (2) [circle,thick,draw,fill=white] at (1,0) {2};
\node (3) [circle,thick,draw,fill=join,text=white] at (2,0) {3};
\node (4) [circle,thick,draw,fill=meet] at (3,0) {4};
\draw[->,thick] (2) to (1);
\draw[->,thick](2) to (3);
\draw[->,thick] (4) to (3);
\end{tikzpicture}}};
\node (e) [align=center] at (1,3) {\scalebox{0.4}{\begin{tikzpicture}
\node (1) [circle,thick,draw,fill=join,text=white] at (0,0) {1};
\node (2) [circle,thick,draw,fill=meet] at (1,0) {2};
\node (3) [circle,thick,draw,fill=white] at (2,0) {3};
\node (4) [circle,thick,draw,fill=join,text=white] at (3,0) {4};
\draw[->,thick] (2) to (1);
\draw[->,thick] (2) to (3);
\draw[->,thick] (4) to (3);
\end{tikzpicture}}};
\node (g) [align=center] at (-1,3) {\scalebox{0.4}{\begin{tikzpicture}
\node (1) [circle,thick,draw,fill=white] at (0,0) {1};
\node (2) [circle,thick,draw,fill=join,text=white] at (1,0) {2};
\node (3) [circle,thick,draw,fill=white] at (2,0) {3};
\node (4) [circle,thick,draw,fill=meet] at (3,0) {4};
\draw[->,thick] (2) to (1);
\draw[->,thick] (2) to (3);
\draw[->,thick] (4) to (3);
\end{tikzpicture}}};
\node (i) [align=center] at (0,4) {\scalebox{0.4}{\begin{tikzpicture}
\node (1) [circle,thick,draw,fill=white] at (0,0) {1};
\node (2) [circle,thick,draw,fill=join,text=white] at (1,0) {2};
\node (3) [circle,thick,draw,fill=white] at (2,0) {3};
\node (4) [circle,thick,draw,fill=join,text=white] at (3,0) {4};
\draw[->,thick] (2) to (1);
\draw[->,thick] (2) to (3);
\draw[->,thick] (4) to (3);
\end{tikzpicture}}};
\draw[-,thick] (a) to (b) to (f) to (g) to (i);
\draw[-,thick] (a) to (d) to (c) to (e) to (i);
\draw[-,thick] (d) to (f) to (e);
\end{tikzpicture}} \hfill
\raisebox{-0.5\height}{\begin{tikzpicture}[scale=1.4]
\node (a) [align=center] at (0,0) {\scalebox{0.4}{\begin{tikzpicture}
\node (1) [circle,thick,draw,fill=white] at (0,0) {1};
\node (2) [circle,thick,draw,fill=meet] at (1,0) {2};
\node (3) [circle,thick,draw,fill=white] at (2,0) {3};
\node (4) [circle,thick,draw,fill=meet] at (3,0) {4};
\draw[->,thick] (1) to (2);
\draw[->,thick] (3) to (2);
\draw[->,thick] (4) to (3);
\end{tikzpicture}}};
\node (b) [align=center] at (-1,1) {\scalebox{0.4}{\begin{tikzpicture}
\node (1) [circle,thick,draw,fill=white] at (0,0) {1};
\node (2) [circle,thick,draw,fill=meet] at (1,0) {2};
\node (3) [circle,thick,draw,fill=white] at (2,0) {3};
\node (4) [circle,thick,draw,fill=join,text=white] at (3,0) {4};
\draw[->,thick] (1) to (2);
\draw[->,thick] (3) to (2);
\draw[->,thick] (4) to (3);
\end{tikzpicture}}};
\node (c) [align=center] at (2,2) {\scalebox{0.4}{\begin{tikzpicture}
\node (1) [circle,thick,draw,fill=join,text=white] at (0,0) {1};
\node (2) [circle,thick,draw,fill=white] at (1,0) {2};
\node (3) [circle,thick,draw,fill=meet] at (2,0) {3};
\node (4) [circle,thick,draw,fill=white] at (3,0) {4};
\draw[->,thick] (1) to (2);
\draw[->,thick] (3) to (2);
\draw[->,thick] (4) to (3);
\end{tikzpicture}}};
\node (f) [align=center] at (0,2) {\scalebox{0.4}{\begin{tikzpicture}
\node (1) [circle,thick,draw,fill=meet] at (0,0) {1};
\node (2) [circle,thick,draw,fill=white] at (1,0) {2};
\node (3) [circle,thick,draw,fill=join,text=white] at (2,0) {3};
\node (4) [circle,thick,draw,fill=meet] at (3,0) {4};
\draw[->,thick] (1) to (2);
\draw[->,thick] (3) to (2);
\draw[->,thick] (4) to (3);
\end{tikzpicture}}};
\node (d) [align=center] at (1,1) {\scalebox{0.4}{\begin{tikzpicture}
\node (1) [circle,thick,draw,fill=meet] at (0,0) {1};
\node (2) [circle,thick,draw,fill=join,text=white] at (1,0) {2};
\node (3) [circle,thick,draw,fill=meet] at (2,0) {3};
\node (4) [circle,thick,draw,fill=white] at (3,0) {4};
\draw[->,thick] (1) to (2);
\draw[->,thick] (3) to (2);
\draw[->,thick] (4) to (3);
\end{tikzpicture}}};
\node (e) [align=center] at (1,3) {\scalebox{0.4}{\begin{tikzpicture}
\node (1) [circle,thick,draw,fill=join,text=white] at (0,0) {1};
\node (2) [circle,thick,draw,fill=white] at (1,0) {2};
\node (3) [circle,thick,draw,fill=join,text=white] at (2,0) {3};
\node (4) [circle,thick,draw,fill=meet] at (3,0) {4};
\draw[->,thick] (1) to (2);
\draw[->,thick] (3) to (2);
\draw[->,thick] (4) to (3);
\end{tikzpicture}}};
\node (g) [align=center] at (-1,3) {\scalebox{0.4}{\begin{tikzpicture}
\node (1) [circle,thick,draw,fill=meet] at (0,0) {1};
\node (2) [circle,thick,draw,fill=join,text=white] at (1,0) {2};
\node (3) [circle,thick,draw,fill=white] at (2,0) {3};
\node (4) [circle,thick,draw,fill=join,text=white] at (3,0) {4};
\draw[->,thick] (1) to (2);
\draw[->,thick] (3) to (2);
\draw[->,thick] (4) to (3);
\end{tikzpicture}}};
\node (i) [align=center] at (0,4) {\scalebox{0.4}{\begin{tikzpicture}
\node (1) [circle,thick,draw,fill=join,text=white] at (0,0) {1};
\node (2) [circle,thick,draw,fill=white] at (1,0) {2};
\node (3) [circle,thick,draw,fill=white] at (2,0) {3};
\node (4) [circle,thick,draw,fill=join,text=white] at (3,0) {4};
\draw[->,thick] (1) to (2);
\draw[->,thick] (3) to (2);
\draw[->,thick] (4) to (3);
\end{tikzpicture}}};
\draw[-,thick] (a) to  (b);
\draw[-,thick] (f) to (e) to  (i);
\draw[-,thick] (b) to (g);
\draw[-,thick] (d) to (f) to (g);
\draw[-,thick] (d) to (c) to (e);
\draw[-,thick] (g) to (i);
\draw[-,thick] (a) to (d);
\end{tikzpicture}}
\caption{Independence posets for four orientations of a path of length $4$.  Each poset has eight elements (the tops drawn in blue and orange as in~\Cref{fig:first_examples,fig:a_toggle}), corresponding to the eight independent sets in the underlying undirected graph.  The top left poset is not a lattice, the bottom left poset is a distributive lattice, and both posets on the right are trim lattices.}
\label{fig:examples}
\end{figure}

\subsection{Independence lattices are trim}

A \defn{trim lattice} is an extremal left-modular lattice~\cite{thomas2006analogue}.  Trim lattices were introduced to serve as analogues of distributive lattices without the graded hypothesis: a graded trim lattice is a distributive lattice, and every distributive lattice is trim.  

Our independence posets further generalize distributive lattices by removing the lattice requirement: an independence poset that is a lattice is always a trim lattice, and every trim lattice can be realized as an independence poset for a unique (up to isomorphism) acyclic directed graph $G$.  In other words, the common intersection of lattices and independence posets are exactly the trim lattices.

\begin{theorem}
\label{thm:lattice_trim}
If $\trip(G)$ is a lattice, then it is a trim lattice.
\end{theorem}

Following Markowsky~\cite{markowsky1975factorization,markowsky1992primes,TW}, a \defn{{\bf m}aximal {\bf o}rthogonal {\bf p}air} (or mop) in an acyclic directed graph $G$ is a pair of sets $(X,Y)$ such that no edges run from $X$ to $Y$, and such that $X$ and $Y$ are both maximal with respect to this condition.  Markowsky's generalization of Birkhoff's fundamental theorem of finite distributive lattices states that any extremal lattice has a representation $\LL(G)$ as the lattice of maximal orthogonal pairs of a unique (up to isomorphism) acyclic directed graph $G$.

\begin{theorem}
If $\LL(G)$ is trim, then $\LL(G) \simeq \trip(G)$.
\label{thm:trim_eq}
\end{theorem}

Moreover, if $\trip(G)$ is a lattice, then $\trip(G) \simeq \LL(G)$.  We provide explicit bijections between tops and mops in~\Cref{sec:bijections}.  The common lattice of tops and mops offer different advantages.  Cover relations $x \lessdot y \in \trip(G)$ are easy to compute using tops (by flips), but harder to see using mops.  Similarly, relations $x<y \in \LL(G)$ are easy to compute using mops (by inclusion), but harder to see using tops.

\Cref{thm:five_sets} characterizes when $\trip(G)$ is a lattice---or, equivalently, when $\LL(G)$ is trim---in terms of graph-theoretic properties of $G$.  We discuss some properties of independence posets that are not lattices in~\Cref{sec:nonlattices}.

\subsection{Toggles}
By~\Cref{thm:trip_ind}, the number of elements of $\trip(G)$ is equal to the number of independent sets in the \emph{undirected} graph $G$.  For $g$ a minimal or maximal element of $G$, there is a natural \defn{toggle} operation $\mut_g$ (similar to quiver mutation) that  reverses every edge incident to $g$; this operation induces a bijection between $\trip(G)$ and $\trip(\mut_g(G))$ whose effect essentially interchanges the relative order of a decomposition of $\trip(G)$ into two intervals using the element $g$.

By limiting ourselves to only toggling at maximal elements (or at minimal elements) and only keeping track of the sets $\br$ (or $\al$), this bijection can be computed cleanly at the level of independent sets (see~\Cref{eq:mut2} and ~\Cref{thm:mutation_bijection}).
Some examples of toggles on an orientation of the path of length three are given in~\Cref{fig:toggle}.  See also~\cite{striker2012promotion,striker2016rowmotion}.

\subsection{Rowmotion}

Since both components of a top are independent sets, and each independent set can be completed to a top in two ways, it is natural to define \defn{rowmotion} by sending one completion to the other:
\begin{equation}\row(\al,\br):=\text{the unique } (\al',\br') \in \trip(G) \text{ with } \al=\br'.\label{eq:global_row}\end{equation}

It turns out that there are two other equally natural (but slower) ways to compute rowmotion.   The distinction between these two slower methods was not apparent when rowmotion had been studied at the level of distributive lattices~\cite{fon1993orbits,cameron1995orbits,striker2012promotion}, but our generalized setting of independence posets reveals their differences: one method computes rowmotion as a composition of flips within a \emph{fixed} independence poset (\defn{rowmotion in slow motion}), while the second relies on a sequence of toggles and the corresponding bijections between independence posets for different orientations of the same underlying undirected graph (\defn{rowmotion by deformotion}).

\begin{theorem}
Let $G$ be a directed acyclic graph.  Then rowmotion can be computed in slow motion and by deformotion---that is, \[\row = \prod_{g \in \ell} \tog_{g} = \prod_{g \in \ell'} \mut_{g}\] for any linear extension $\ell$ and reverse linear extension $\ell'$ of $G$-order.
\label{thm:main_thm}
\end{theorem}

\subsection{Representation theory}
We conclude with some applications to representation theory.  Let $k$ be a field, and $A$ a finite-dimensional $k$-algebra such that the module category $\mod A$ has no cycles.  Define a directed graph $G$ with vertices indexed by the indecomposable $A$-modules and an arrow from $M$ to $N$ if and only if $\Hom(M,N)\ne 0$.  By our assumption that there are no cycles in
$\mod A$, the graph $G$ is acyclic---but not all acyclic directed graphs $G$ arise in this way.

A \defn{torsion class} in $\mod A$ is a full additive
subcategory of $\mod A$ closed under
extensions and quotients.  A module is called a
\defn{brick} if its endomorphism ring is a division algebra.  A collection of bricks is called a \defn{semibrick} if there are no morphisms between two
non-isomorphic bricks in the collection.  Finally, \defn{2-simple-minded
collections} (defined in~\Cref{sec:2simple}) are certain collections of objects in the derived category of $A$ in bijection with torsion classes~\cite{asai2016semibricks}.

\begin{theorem}
\label{thm:reptheory}
If $A$ is representation finite and $\mod A$ has no cycles, then
maximal orthogonal pairs correspond to torsion pairs, independent sets correspond to semibricks, and tight orthogonal pairs correspond to 2-simple-minded collections.
\end{theorem}

\subsection{Organization of the paper}

In~\Cref{sec:top} we complete independent sets to tight orthogonal pairs with~\Cref{map:oneG,map:zeroG}, proving~\Cref{thm:trip_ind}.  In~\Cref{sec:flip}, we define flips with~\Cref{map:flip}, and prove they are well-defined in~\Cref{prop:flips_work}.  We study independence posets in~\Cref{sec:independence_posets}, and present a useful recursion on tops in~\Cref{sec:recursion}.  We relate independence posets to trim lattices in~\Cref{sec:trim_lattices}, relating tight orthogonal pairs to maximal orthogonal pairs and proving~\Cref{thm:lattice_trim,thm:trim_eq}.  In~\Cref{sec:toggles}, we define toggles on independence posets; we then study rowmotion in~\Cref{sec:rowmotion}, proving~\Cref{thm:main_thm}.  We conclude with connections to representation theory in~\Cref{sec:rep_theory}.

\section{Tight Orthogonal Pairs}
\label{sec:top}

\begin{definition}
Let $g \in G$ and define $G_g$ to be the directed graph obtained by deleting the vertex $g$ from $G$ (along with all edges to $g$), and  $G_g^\circ$ the directed graph obtained by deleting all vertices and edges adjacent to $g$ in $G$ (along with $g$ itself).
\end{definition}

{
\renewcommand{\thetheorem}{\ref{thm:trip_ind}}
\begin{theorem} Let $\mathcal{I}$ be an independent set of an acyclic directed graph $G$.  Then there exists a unique $(\mathcal{I},\br) \in \trip(G)$ and a unique $(\al,\mathcal{I})\in \trip(G)$.
\end{theorem}
\addtocounter{theorem}{-1}
}

\begin{proof}
Let $\mathcal{I}$ be an independent set.  We show that~\Cref{map:oneG} produces an element $(\al,\mathcal{I}) \in \trip(G)$ (this algorithm is illustrated in~\Cref{fig:alg_ex23}).  By construction, the output of \Cref{map:oneG} is an orthogonal pair of independent sets.  We claim it is tight.

Suppose the output is not tight.  Then at least one of the following holds:
\begin{itemize} \item there is some element $g \in G$ that could be added to $\al$ to still have an orthogonal pair of independent sets,
\item there is some element $g' \in \al$ that could be increased to $g \in G$ with respect to $G$-order,
\item there is some element $g\in G$ that could be added into $\mathcal{I}$, or
\item there is some element $g' \in \mathcal{I}$ that could be decreased to $g\in G$.
\end{itemize}
Take a maximal $g$ among all such elements.  One verifies that~\Cref{map:oneG} would have added it in, which is a contradiction.  Similar reasoning shows that~\Cref{map:zeroG} produces an element $(\mathcal{I},\br) \in \trip(G)$.

We now show by induction on $|G|$ that given $\mathcal{I}$, the tight orthogonal pair $(\al,\mathcal{I})$ is unique.  The base case for $|G|=1$ is trivial.  We now suppose $(\al,\mathcal{I})$ and $(\al',\mathcal{I})$ are two different tight orthogonal pairs. Let $g$ be minimal in $G$.

If $g \not \in \mathcal{I}$, the restriction of $(\al,\mathcal{I})$ and $(\al',\mathcal{I})$ to $G_g$ are  tight orthogonal pairs of $G_g$; by induction, they must coincide except possibly at $g$.  But since $\al \neq \al'$ then either $\al \subset \al'$ or $\al' \subset \al$, which contradicts tightness.

Otherwise, $g \in \mathcal{I}$.  The restriction of $(\al,\mathcal{I})$ and $(\al',\mathcal{I})$ to  $G_g^\circ$ are now  tight orthogonal pairs of $G_g^\circ$ and coincide.  But by definition, there can be no elements of either $\al$ or $\al'$ adjacent to $g$ and so  $(\al,\mathcal{I})$ and $(\al',\mathcal{I})$ coincide.

The argument that the  tight orthogonal pair $(\mathcal{I},\br)$ produced by~\Cref{map:zeroG} is unique is similar, instead letting $g$ be maximal in $G$.
\end{proof}

\begin{algorithm}[htbp]
\DontPrintSemicolon
\KwIn{An acyclic directed graph $G$ and an independent set $\mathcal{I}$.}
\KwOut{An element $(\al,\mathcal{I}) \in \trip(G)$.}
\kwset{$\al=\{\}$}
\For{$k$ in $\ell'$}{
    \lIf{$\left\{\begin{tabular}{l} $k \not \in \mathcal{I}$\\ $i \to k \not \in G$ for $i \in \al$ \\ $k \to i\not \in G$ for $i \in \mathcal{I}$ \end{tabular}\right\}$}{add $k$ to $\al$}
}
\Return $(\al,\mathcal{I})$\;
\caption{The greedy construction of the unique $(\al,\mathcal{I}) \in \trip(G)$ using any reverse linear extension $\ell'$ of $G$-order, given an independent set $\mathcal{I}$. See~\Cref{fig:alg_ex23} for an example.}
\label{map:oneG}
\end{algorithm}

\begin{algorithm}[htbp]
\DontPrintSemicolon
\KwIn{An acyclic directed graph $G$ and an independent set $\mathcal{I}$.}
\KwOut{An element $(\mathcal{I},\br) \in \trip(G)$.}
\kwset{$\br=\{\}$}
\For{$k$ in $\ell$}{
    \lIf{$\left\{\begin{tabular}{l} $k \not \in \mathcal{I}$\\ $i \to k \not \in G$ for $i \in \mathcal{I}$ \\ $k \to i\not \in G$ for $i \in \br$ \end{tabular}\right\}$}{add $k$ to $\br$}
}
\Return $(\mathcal{I},\br)$\;
\caption{The greedy construction of the unique $(\mathcal{I},\br) \in \trip(G)$ using any linear extension $\ell$ of $G$-order, given an independent set $\mathcal{I}$.  See~\Cref{fig:alg_ex23} for an example.}
\label{map:zeroG}
\end{algorithm}

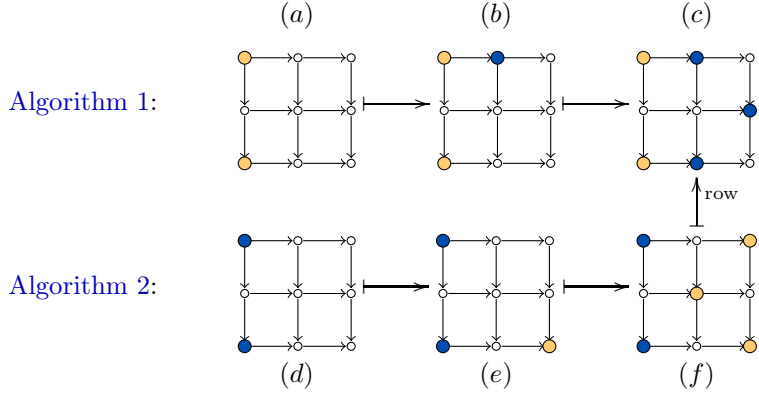
\begin{figure}[htbp]
\xymatrix@R-2pc{
&(a) &(b) & (c)\\
\text{\Cref{map:oneG}:} & \raisebox{-0.5\height}{\begin{tikzpicture}[scale=.7]
 	\foreach \i in {0,...,2}{
	  \foreach \j in {0,...,2}{
          \draw[shape=circle,black,fill=white] (\i,\j) circle (.5ex);
      };
	};
    \foreach \i in {0,...,1}{
	  \foreach \j in {0,...,2}{
          \draw[->] (\i+.1,\j) to (\i+.9,\j);
      };
	};
    \foreach \i in {0,...,2}{
	  \foreach \j in {1,...,2}{
           \draw[->] (\i,\j-.1) to (\i,\j-.9);
      };
	};
    \foreach \i in {(0,0),(0,2)}{
          \draw[shape=circle,fill=meet] \i circle (.8ex);
	};
  \end{tikzpicture}} \ar@{|->}[r] &
  \raisebox{-0.5\height}{\begin{tikzpicture}[scale=.7]
  \foreach \i in {0,...,2}{
	  \foreach \j in {0,...,2}{
          \draw[shape=circle,black,fill=white] (\i,\j) circle (.5ex);
      };
	};
    \foreach \i in {0,...,1}{
	  \foreach \j in {0,...,2}{
          \draw[->] (\i+.1,\j) to (\i+.9,\j);
      };
	};
    \foreach \i in {0,...,2}{
	  \foreach \j in {1,...,2}{
           \draw[->] (\i,\j-.1) to (\i,\j-.9);
      };
	};
    \foreach \i in {(0,0),(0,2)}{
          \draw[shape=circle,fill=meet] \i circle (.8ex);
	};
     \foreach \i in {(1,2)}{
          \draw[shape=circle,fill=join] \i circle (.8ex);
	};
  \end{tikzpicture}}  \ar@{|->}[r] & \raisebox{-0.5\height}{\begin{tikzpicture}[scale=.7]
 	\foreach \i in {0,...,2}{
	  \foreach \j in {0,...,2}{
          \draw[shape=circle,black,fill=white] (\i,\j) circle (.5ex);
      };
	};
    \foreach \i in {0,...,1}{
	  \foreach \j in {0,...,2}{
          \draw[->] (\i+.1,\j) to (\i+.9,\j);
      };
	};
    \foreach \i in {0,...,2}{
	  \foreach \j in {1,...,2}{
           \draw[->] (\i,\j-.1) to (\i,\j-.9);
      };
	};
    \foreach \i in {(0,0),(0,2)}{
          \draw[shape=circle,fill=meet] \i circle (.8ex);
	};
     \foreach \i in {(1,0),(1,2),(2,1)}{
          \draw[shape=circle,fill=join] \i circle (.8ex);
	};
  \end{tikzpicture}}\\ \\ \\ \\
  \text{\Cref{map:zeroG}:} &
  \raisebox{-0.5\height}{\begin{tikzpicture}[scale=.7]
 	\foreach \i in {0,...,2}{
	  \foreach \j in {0,...,2}{
          \draw[shape=circle,black,fill=white] (\i,\j) circle (.5ex);
      };
	};
    \foreach \i in {0,...,1}{
	  \foreach \j in {0,...,2}{
          \draw[->] (\i+.1,\j) to (\i+.9,\j);
      };
	};
    \foreach \i in {0,...,2}{
	  \foreach \j in {1,...,2}{
           \draw[->] (\i,\j-.1) to (\i,\j-.9);
      };
	};
    \foreach \i in {(0,0),(0,2)}{
          \draw[shape=circle,fill=join] \i circle (.8ex);
	};
  \end{tikzpicture}}  \ar@{|->}[r] & \raisebox{-0.5\height}{\begin{tikzpicture}[scale=.7]
 	\foreach \i in {0,...,2}{
	  \foreach \j in {0,...,2}{
          \draw[shape=circle,black,fill=white] (\i,\j) circle (.5ex);
      };
	};
    \foreach \i in {0,...,1}{
	  \foreach \j in {0,...,2}{
          \draw[->] (\i+.1,\j) to (\i+.9,\j);
      };
	};
    \foreach \i in {0,...,2}{
	  \foreach \j in {1,...,2}{
           \draw[->] (\i,\j-.1) to (\i,\j-.9);
      };
	};
    \foreach \i in {(0,0),(0,2)}{
          \draw[shape=circle,fill=join] \i circle (.8ex);
	};
     \foreach \i in {(2,0)}{
          \draw[shape=circle,fill=meet] \i circle (.8ex);
	};
  \end{tikzpicture}}  \ar@{|->}[r] & \raisebox{-0.5\height}{\begin{tikzpicture}[scale=.7]
 	\foreach \i in {0,...,2}{
	  \foreach \j in {0,...,2}{
          \draw[shape=circle,black,fill=white] (\i,\j) circle (.5ex);
      };
	};
    \foreach \i in {0,...,1}{
	  \foreach \j in {0,...,2}{
          \draw[->] (\i+.1,\j) to (\i+.9,\j);
      };
	};
    \foreach \i in {0,...,2}{
	  \foreach \j in {1,...,2}{
           \draw[->] (\i,\j-.1) to (\i,\j-.9);
      };
	};
    \foreach \i in {(0,0),(0,2)}{
          \draw[shape=circle,fill=join] \i circle (.8ex);
	};
     \foreach \i in {(1,1),(2,0),(2,2)}{
          \draw[shape=circle,fill=meet] \i circle (.8ex);
	};
  \end{tikzpicture}} \ar@{|->}[uuuu]_{\row} \\ &(d) &(e) & (f)}
\caption{The progression $(a)\mapsto(b)\mapsto(c)$ is an illustration of~\Cref{map:oneG}, which greedily adds elements to the independent set $(a)$ in a reverse linear extension of $G$-order (from the top left to the bottom right).  The progression $(d)\mapsto(e)\mapsto(f)$ illustrates~\Cref{map:zeroG}, which greedily adds elements to the independent set $(d)$ in a linear extension of $G$-order (from the bottom right to the top left).  Rowmotion, defined by~\Cref{eq:global_row}, sends the tight orthogonal pair $(f)$ on the bottom to the tight orthogonal pair $(c)$.}
\label{fig:alg_ex23}
\end{figure}

\section{Flips and Independence Posets}

\subsection{Flips on tight orthogonal pairs}
\label{sec:flip}
We will define a poset structure on the tight orthogonal pairs of $G$ by specifying the cover relations.   To this end, \Cref{def:flip} and \Cref{map:flip} define a \defn{flip} of a  tight orthogonal pair $(\al,\br)$ at an element $g \in G$, written $\tog_g(\al,\br)$; the flip moves \defn{up} in the poset if $g \in \br$ and moves \defn{down} in the poset if $g \in \al$ (and does nothing otherwise).  \Cref{fig:a_toggle} illustrates a flip on a  tight orthogonal pair in an orientation of $[7]\times [7]$.  We first prove that the image of a flip is again a tight orthogonal pair.

\begin{proposition}
\label{prop:flips_work}
Let $g$ be an element of an acyclic directed graph $G$.  Then $\tog_g(\al,\br)\in \trip(G)$.
\end{proposition}

\begin{proof}
The statement follows from the restriction of~\Cref{thm:trip_ind} to the elements of $G$ not less than $g$ and to the elements not greater than $g$.
\end{proof}

\begin{algorithm}[htbp]
\DontPrintSemicolon
\KwIn{$(\al,\br) \in \trip(G)$ and $g \in G$.}
\KwOut{$(\al',\br') \in \trip(G)$.}
\kwset{$\al'=\{k : k \in \al \text{ and } k \not \leq g\}$ and $\br'=\{k : k \in \br \text{ and } k \not \geq g\}$\;}
\lIf{$g \not \in \br$ and $g \not \in \al$}{\Return $(\al,\br)$}
\lElseIf{$g \in \br$}{$\al'=\al'\cup\{g\}$}
\lElseIf{$g \in \al$}{$\br'=\br'\cup\{g\}$}
\For{$k$ in $\ell'$}{
    \lIf{$\left\{\begin{tabular}{l}$k \not \geq g$, $k \not \in \br$ \\$k\to i\not \in G$ for $i \in \br$ \\$i \to k\not \in G$ for $i\in \al$  \end{tabular}\right\}$}{add $k$ to $\al'$}
}
\For{$k$ in $\ell$}{
    \lIf{$\left\{\begin{tabular}{l}$k \not \leq g$, $k \not \in \al$\\ $i\to k \not \in G$ for $i \in \al$\\ $k \to i\not \in G$ for $i\in \br$\end{tabular}\right\}$}{add $k$ to $\br'$}
}
\Return $(\al',\br')$\;
\caption{The definition of a flip of $(\al,\br) \in \trip(G)$ at an element $g \in G$, written $\tog_g(\al,\br)$.  As usual, $\ell$ is a linear extension of $G$-order, while $\ell'$ is a reverse linear extension.}
\label{map:flip}
\end{algorithm}

\begin{lemma}  Let $g$ be an element of an acyclic directed graph $G$.  Then \[\tog_g^2(\al,\br) = (\al,\br).\]  If $h$ is incomparable with $g$ in $G$-order, then $\tog_g \circ \tog_h = \tog_h \circ \tog_g$.
\label{lem:flips_invs}
\end{lemma}

\begin{proof}
For the first statement, a flip preserves the elements of $\al$ that are not less than $g$, and the elements of $\br$ that are not greater than $g$.  Since $g$ has now returned to its original set after flipping twice, the restriction of~\Cref{thm:trip_ind} to the elements of $G$ not greater than $g$ show that the preserved elements of $\br$ force the recovery of the elements of $\al$  less than $g$, and similarly that the preserved elements of $\al$ force the recovery of the elements of $\br$ greater than $g$.  (And all other elements of $\al$ and $\br$ weren't affected by the flip).  The second statement is immediate due to the order in which vertices are added to $\al$ and $\br$.
\end{proof}

\subsection{Independence posets}
\label{sec:independence_posets}

For $G$ an acyclic directed graph, the \defn{independence relations} on $\trip(G)$ are the reflexive and transitive closure of the relations $(\al,\br) < (\al',\br')$ if there is some $g \in \br$ such that $\tog_g(\al,\br)=(\al',\br')$.  

\begin{lemma} Independence relations are antisymmetric, and hence define an \defn{independence poset}, denoted $\trip(G)$.  Flips and cover relations of $\trip(G)$ coincide.\label{lem:flip_eq_cov} \end{lemma}
\begin{proof}
Note that the relation is antisymmetric: any upwards flip from $(\al,\br)$ at $g$
introduces the new element $g$ into $\al$, and the only way to remove it by a subsequent upwards flip is to flip at an element higher than $g$, which then introduces a new element not in $\al$.  No such sequence of flips can therefore terminate at $(\al,\br)$, and so the relation is antisymmetric.

By the definition of $\trip(G)$, every cover relation of $G$
is induced by a flip.
Suppose now that there is some $(\al,\br)\in\trip(G)$, and some $g$
in $\br$, so that $(\al',\br')=\tog_g(\al,\br)>(\al,\br)$.
What has to be checked is that there is no longer sequence of upward flips which
also interpolates between $(\al,\br)$ and $(\al',\br')$.  Any upwards flip from $(\al,\br)$ at $h$ not below
$g$ would introduce the new element $h$ into $\al$, which is not in $\al'$; the only way to remove it by a subsequent upwards flip is to flip at an element higher than $h$, which then introduces a new element not in $\al'$ into $\al$.  No such sequence of flips can therefore terminate at $(\al',\br')$.  Dually, any
upwards flip at $h$ not below $g$ removes an element from $\br$ which is contained in $\br'$, and the only way to restore it is to flip at a still lower element, which removes a different element of $\br'$ from $\br$.  Similarly,
therefore, no such sequence of upwards flips can therefore terminate at $(\al',\br')$.
\end{proof}

By~\Cref{def:independence_poset}, the maximum element of $\trip(G)$ is the unique  tight orthogonal pair $\hat{1}$ of the form $(\al,\emptyset)$, and its minimum element $\hat{0}$ is of the form $(\emptyset,\br)$.  In particular, we see that $\trip(G)$ is connected.  \Cref{fig:toggle} gives several examples of independence posets on various orientations of a path of length 3.

A \defn{chain} in a poset is a sequence of elements $\x_0 < \x_1 < \cdots < \x_r,$ of \defn{length} $r$.  The poset $\trip(G)$ has a maximal chain of length $|G|$ obtained by starting at $\hat{0}$ and flipping the elements of $G$ in the order of a linear extension of $G$-order.

\begin{lemma}
\label{lem:max_chain}
For $G$ an acyclic directed graph with $g_1,\ldots,g_{|G|}$ a linear extension of $G$-order, the sequence \[\hat{0} \lessdot \tog_{g_1}(\hat{0}) \lessdot (\tog_{g_2}\circ \tog_{g_1})(\hat{0})\lessdot \ldots \lessdot (\tog_{g_{|G|}}\circ\cdots\circ \tog_{g_1})(\hat{0})=\hat{1}\] is a maximal chain in $\trip(G)$.
\end{lemma}

\begin{proof}
Write $(\al_i,\br_i)$ for the $i$th element of the sequence.  By~\Cref{lem:flip_eq_cov}, this sequence is unrefinable.  By induction, after the $i$th step all elements of $\br_i$ lie above $\{g_1,\ldots,g_i\}$ and all elements of $\al_i$ are contained in $\{g_1,\ldots,g_i\}$.  Furthermore, since each $\br_i$ is computed greedily in linear extension order, $g_{i+1} \in \br_i$.  The sequence must end with $\hat{1}$, because the only way for all elements of $\br_{|G|}$ to lie above $\{g_1,\ldots,g_{|G|}\}$ is for $\br_{|G|}$ to be empty.
\end{proof}

\begin{lemma}
Fix a linear extension $\ell$ of $G$-order.  For any element $(\al,\br) \in \trip(G)$, there is a unique chain
 \[\hat{0} \lessdot \tog_{g_1}(\hat{0}) \lessdot (\tog_{g_2}\circ \tog_{g_1})(\hat{0})\lessdot \ldots \lessdot (\tog_{g_{k}}\circ\cdots\circ \tog_{g_1})(\hat{0})=(\al,\br)\] such that $g_1 <_{\ell} g_2 <_\ell \cdots <_\ell g_k$.
\label{lem:tree_structure}
\end{lemma}

\begin{proof}
Starting with $(\al,\br)$, flip  elements $h$ out of $\al$ in reverse linear extension order.  This does not add elements larger than $h$ into $\al$.  The process must therefore terminate with $\hat{0}$.  Uniqueness follows from the fact that if we ever flip a lower element of $\al$ than prescribed above, we will never be able to remove from $\al$ the elements that we skipped over (without violating the constraint on the order of the flips).
\end{proof}

{
\renewcommand{\thetheorem}{\ref{thm:tree_structure}}
\begin{theorem}
Fix a linear extension $\ell$ of $G$-order.  Flipping only in increasing order of $\ell$ gives a tree structure on the independent sets of $G$.
\end{theorem}
\addtocounter{theorem}{-1}
}

\begin{proof}
\Cref{lem:tree_structure} shows that by only permitting flips in the order of some fixed linear extension $\ell$ of $G$-order, we obtain a spanning tree of $\trip(G)$.
\end{proof}

This tree structure is illustrated in~\Cref{fig:tamari} for the Tamari lattice on 14 elements.

\begin{proposition}
For $G$ a directed acyclic graph, let $G^*$ be the graph obtained by reversing all the edges in $G$.  Then \begin{align*}\trip(G) &\simeq \trip(G^*)^* \\ (\al,\br) &\mapsto (\br,\al), \end{align*} where $\trip(G^*)^*$ is the poset dual of $\trip(G^*)$.
\label{prop:dual}
\end{proposition}
\begin{proof}
Immediate.
\end{proof}

\subsection{Tight orthogonal pair recursion}
\label{sec:recursion}
For any $g \in G$, since $\{g\}$ is an independent set of $G$, by~\Cref{thm:trip_ind} there is a unique  tight orthogonal pair $m_g$ of the form $(\al,\{g\})$, and a unique  tight orthogonal pair $j_g$ of the form $(\{g\},\br)$.  Write \begin{equation}\trip_g(G):=[\hat{0},m_g]  \text{ and } \trip^g(G):=[j_g,\hat{1}].\end{equation}  We say that $g \in G$ is \defn{extremal} if it is a minimal or maximal element of $G$-order.

\begin{lemma}
\label{lem:decomposition}
Let $G$ be an acyclic directed graph.  If $g$ is an extremal element of $G$, then $\trip(G) = \trip_g(G) \sqcup \trip^g(G)$ .  Furthermore,
\begin{itemize}
\item If $g$ is minimal, $(\al,\br) \in \trip_g(G)$ if and only if $g \in \br$, and \item If $g$ is maximal, $(\al,\br) \in \trip^g(G)$ if and only if $g \in \al$.
\end{itemize}
In particular, if $x \in \trip^g(G)$ and $y \in \trip_g(G)$, then $x \not \leq y$.
\end{lemma}

\begin{proof}
Suppose $g$ is a minimal element of $G$ and let $(\al,\br) \in \trip(G)$.

If $g \in \br$, then by successively flipping in any order all elements not equal to $g$ that only cause us to move up in $\trip(G)$, we must end with the element $m_g$ and so $(\al,\br) \in \trip_g(G)$.  For certainly $g \in \br$ since it started out in $\br$ and flipping at elements of $G$ not equal to $g$ does not remove $g$ from $\br$, since $g$ is minimal.  And $g$ is the only element of $\br$, or we would have flipped more elements out.

If $(\al,\br) \in \trip_g(G)$, then there is a sequence of flips that take us upwards to $m_g$.  This sequence cannot remove $g$ (since it could never be replaced), and so $g \in \br$.

If $g \in \al$, or if $g \not \in \al\cup \br$, then we wish to successively flip in any order all elements not equal to $g$ that only cause us to move down in $\trip(G)$.  We claim that this must end with the element $j_g$ so that $(\al,\br) \in \trip^g(G)$.  Since the poset is finite, the process ends.  Let $(\al',\br')$ be the end of this flip sequence.  Then we claim $g \in \al'$ at this point.  Note that $\al'$ cannot contain any element $h \in G$ adjacent to $g$ (or we would have flipped $h$ out), so that $g$ can be added to $\al'$---unless $g \in \br'$.  But $g$ cannot be in $\br'$, since it was not in $\br$, and only toggling at $g$ could add it in.

The dual argument applies when $g$ is maximal.
\end{proof}

\begin{lemma}
For any element $g \in G$, if $g \in \al$ then $(\al,\br) \in \trip^g(G)$; and if $g \in \br$ then $(\al,\br)\in \trip_g(G)$.
\label{lem:extremal_needs}
\end{lemma}

\begin{proof}
If $g \in \al$, then we construct a path downwards in $\trip(G)$ from $(\al,\br)$ to $j_g$.  First, flip out all elements greater than $g$ in reverse linear extension order. This does not cause $g$ to enter into $\br$ (since all elements in $\br$ less than $h$ are fixed by flips).
Since $g$ can be added to $\al$ at the end of this sequence of flips, it is indeed in $\al$ (though it may not have stayed in it all the time during the sequence).  Now flip out all the remaining elements of $\al$ other than $g$. None of the remaining elements is above $g$, so this will neither remove $g$ nor introduce elements in $\al$ above $g$.
The process necessarily terminates with $j_g$.  Again, the statement for $\br$ follows dually.
\end{proof}

\begin{lemma}
Let $G$ be a directed acyclic graph.
\begin{itemize}
\item If $g$ is minimal and $(\al,\br) \in \trip_g(G)$, then $\tog_g(\al,\br)=(\al\cup \{g\},\br')$ for some $\br'$.
\item If $g$ is maximal and $(\al,\br) \in \trip^g(G)$, then $\tog_g(\al,\br)=(\al',\br\cup\{g\})$ for some $\al'$.
\end{itemize}
\label{lem:down}
\end{lemma}

\begin{proof}
This follows from the definition of flip; when $g$ is minimal, all of $\al$ is preserved since every element of $\al$ is not less than $g$.  Similarly, when $g$ is maximal, all of $\br$ is preserved, since every element of $\br$ is not greater than $g$.
\end{proof}

\begin{theorem}
\label{thm:decomposition}
Let $g$ be an extremal element of an acyclic directed graph $G$.  Then \begin{align*} (\al,\br)&\mapsto (\al ,\br \setminus \{g\}\} \text{ is a bijection } \begin{cases} \trip_g(G)\simeq\trip(G_g^\circ) & \text{if } g \text{ minimal} \\ \trip_g(G)\simeq\trip(G_g) & \text{if } g \text{ maximal} \end{cases} \\
  (\al,\br) &\mapsto (\al \setminus \{g\},\br\} \text{ is a bijection } \begin{cases} \trip^g(G)\simeq\trip(G_g) & \text{if } g \text{ minimal} \\ \trip^g(G)\simeq\trip(G_g^\circ) & \text{if } g \text{ maximal} \end{cases}\end{align*}
\end{theorem}

\begin{proof}
We only prove the results for $g$ minimal, the case for $g$ maximal being analogous.  We first show $\trip_g(G)\simeq\trip(G_g^\circ)$; by~\Cref{lem:decomposition}, $(\al,\br) \in \trip_g(G)$ if and only if $g \in \br$.  But since $(\al,\br)$ is a tight orthogonal pair, no element of $\al$ or of $\br$ can be adjacent to $g$, from which we conclude the result by definition of $G_g^\circ$.  We now show $\trip^g(G)\simeq\trip(G_g)$; since $\trip(G) = \trip_g(G) \sqcup \trip^g(G)$, by~\Cref{lem:decomposition} elements of $\trip^g(G)$ consist of those  tight orthogonal pairs of $G$ with either $g \in \al$ or $g \not \in \br \cup \al$.  Each tight orthogonal pair of $G_g$ can be uniquely extended to such a tight orthogonal pair.
\end{proof}

\section{Trim Lattices and Maximal Orthogonal Pairs}
\label{sec:trim_lattices}
\subsection{Extremal lattices}
An \defn{extremal lattice} is a lattice whose longest chain is of length equal to the number of its join irreducible elements and to the number of its meet irreducible elements.  As motivation for our main result of this section, we have the following easy statement (we will refine it in~\Cref{thm:lattice_trim}).

\begin{lemma}
If $\trip(G)$ is a lattice, then it is an extremal lattice.
\label{lem:lattice_implies_extremal}
\end{lemma}

\begin{proof}
Note that a lattice with a chain of length $n$ (i.e. with $n+1$ elements) must have at least $n$ join-irreducible elements and at least $n$ meet-irreducible elements (since each element of the chain is the join of the join-irreducibles beneath it and the meet of the meet-irreducibles above it).  Suppose $\trip(G)$ is a lattice; since it only has $|G|$ join and $|G|$ meet irreducible elements, and since it has a chain of length $|G|$ by~\Cref{lem:max_chain}, it is extremal.
\end{proof}

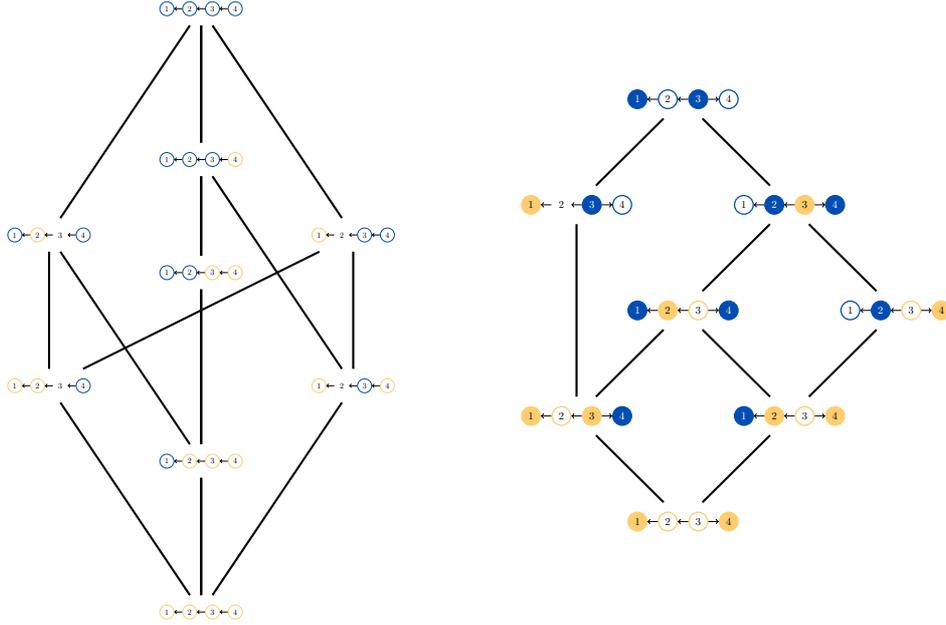
\begin{figure}[htbp]
\raisebox{-0.5\height}{\begin{tikzpicture}[scale=2]
\node (a) [align=center] at (0,0) {\scalebox{0.3}{\begin{tikzpicture}
\node (1) [circle,very thick,draw,color=meet,text=black] at (0,0) {1};
\node (2) [circle,very thick,draw,color=meet,text=black] at (1,0) {2};
\node (3) [circle,very thick,draw,color=meet,text=black] at (2,0) {3};
\node (4) [circle,very thick,draw,color=meet,text=black] at (3,0) {4};
\draw[->,thick] (2) to (1);
\draw[->,thick] (3) to (2);
\draw[->,thick] (4) to (3);
\end{tikzpicture}}};
\node (b) [align=center] at (-1,1.5) {\scalebox{0.3}{\begin{tikzpicture}
\node (1) [circle,very thick,draw,color=meet,text=black] at (0,0) {1};
\node (2) [circle,very thick,draw,color=meet,text=black] at (1,0) {2};
\node (3) [circle,very thick,draw,color=white,text=black] at (2,0) {3};
\node (4) [circle,very thick,draw,color=join,text=black] at (3,0) {4};
\draw[->,thick] (2) to (1);
\draw[->,thick] (3) to (2);
\draw[->,thick] (4) to (3);
\end{tikzpicture}}};
\node (c) [align=center] at (0,1) {\scalebox{0.3}{\begin{tikzpicture}
\node (1) [circle,very thick,draw,color=join,text=black] at (0,0) {1};
\node (2) [circle,very thick,draw,color=meet,text=black] at (1,0) {2};
\node (3) [circle,very thick,draw,color=meet,text=black] at (2,0) {3};
\node (4) [circle,very thick,draw,color=meet,text=black] at (3,0) {4};
\draw[->,thick] (2) to (1);
\draw[->,thick] (3) to (2);
\draw[->,thick] (4) to (3);
\end{tikzpicture}}};
\node (d) [align=center] at (1,1.5) {\scalebox{0.3}{\begin{tikzpicture}
\node (1) [circle,very thick,draw,color=meet,text=black] at (0,0) {1};
\node (2) [circle,very thick,draw,color=white,text=black] at (1,0) {2};
\node (3) [circle,very thick,draw,color=join,text=black] at (2,0) {3};
\node (4) [circle,very thick,draw,color=meet,text=black] at (3,0) {4};
\draw[->,thick] (2) to (1);
\draw[->,thick] (3) to (2);
\draw[->,thick] (4) to (3);
\end{tikzpicture}}};
\node (e) [align=center] at (-1,2.5) {\scalebox{0.3}{\begin{tikzpicture}
\node (1) [circle,very thick,draw,color=join,text=black] at (0,0) {1};
\node (2) [circle,very thick,draw,color=meet,text=black] at (1,0) {2};
\node (3) [circle,very thick,draw,color=white,text=black] at (2,0) {3};
\node (4) [circle,very thick,draw,color=join,text=black] at (3,0) {4};
\draw[->,thick] (2) to (1);
\draw[->,thick] (3) to (2);
\draw[->,thick] (4) to (3);
\end{tikzpicture}}};
\node (f) [align=center] at (0,2.25) {\scalebox{0.3}{\begin{tikzpicture}
\node (1) [circle,very thick,draw,color=join,text=black] at (0,0) {1};
\node (2) [circle,very thick,draw,color=join,text=black] at (1,0) {2};
\node (3) [circle,very thick,draw,color=meet,text=black] at (2,0) {3};
\node (4) [circle,very thick,draw,color=meet,text=black] at (3,0) {4};
\draw[->,thick] (2) to (1);
\draw[->,thick] (3) to (2);
\draw[->,thick] (4) to (3);
\end{tikzpicture}}};
\node (g) [align=center] at (1,2.5) {\scalebox{0.3}{\begin{tikzpicture}
\node (1) [circle,very thick,draw,color=meet,text=black] at (0,0) {1};
\node (2) [circle,very thick,draw,color=white,text=black] at (1,0) {2};
\node (3) [circle,very thick,draw,color=join,text=black] at (2,0) {3};
\node (4) [circle,very thick,draw,color=join,text=black] at (3,0) {4};
\draw[->,thick] (2) to (1);
\draw[->,thick] (3) to (2);
\draw[->,thick] (4) to (3);
\end{tikzpicture}}};
\node (h) [align=center] at (0,3) {\scalebox{0.3}{\begin{tikzpicture}
\node (1) [circle,very thick,draw,color=join,text=black] at (0,0) {1};
\node (2) [circle,very thick,draw,color=join,text=black] at (1,0) {2};
\node (3) [circle,very thick,draw,color=join,text=black] at (2,0) {3};
\node (4) [circle,very thick,draw,color=meet,text=black] at (3,0) {4};
\draw[->,thick] (2) to (1);
\draw[->,thick] (3) to (2);
\draw[->,thick] (4) to (3);
\end{tikzpicture}}};
\node (i) [align=center] at (0,4) {\scalebox{0.3}{\begin{tikzpicture}
\node (1) [circle,very thick,draw,color=join,text=black] at (0,0) {1};
\node (2) [circle,very thick,draw,color=join,text=black] at (1,0) {2};
\node (3) [circle,very thick,draw,color=join,text=black] at (2,0) {3};
\node (4) [circle,very thick,draw,color=join,text=black] at (3,0) {4};
\draw[->,thick] (2) to (1);
\draw[->,thick] (3) to (2);
\draw[->,thick] (4) to (3);
\end{tikzpicture}}};
\begin{pgfonlayer}{background}
\draw[-,thick] (a) to  (b) to  (e) to  (i);
\draw[-,thick] (a) to  (c) to  (e);
\draw[-,thick] (b) to (g) to (i);
\draw[-,thick] (a) to (d) to (h);
\draw[-,thick] (c) to (f) to (h) to (i);
\draw[-,thick] (d) to  (g);
\draw[-,thick] (a) to (c) to  (f) to (h) to (i);
\end{pgfonlayer}
\end{tikzpicture}} \hfill
\raisebox{-0.5\height}{\begin{tikzpicture}[scale=1.4]
\node (a) [align=center] at (0,0) {\scalebox{0.4}{\begin{tikzpicture}
\node (1) [circle,very thick,draw,color=meet,text=black,fill=meet] at (0,0) {1};
\node (2) [circle,very thick,draw,color=meet,text=black] at (1,0) {2};
\node (3) [circle,very thick,draw,color=meet,text=black] at (2,0) {3};
\node (4) [circle,very thick,draw,color=meet,text=black,fill=meet] at (3,0) {4};
\draw[->,thick]  (2) to (1);
\draw[->,thick] (3) to (2);
\draw[->,thick] (3) to (4);
\end{tikzpicture}}};
\node (b) [align=center] at (-1,1) {\scalebox{0.4}{\begin{tikzpicture}
\node (1) [circle,very thick,draw,color=meet,text=black,fill=meet] at (0,0) {1};
\node (2) [circle,very thick,draw,color=meet,text=black] at (1,0) {2};
\node (3) [circle,very thick,draw,color=meet,text=black,fill=meet] at (2,0) {3};
\node (4) [circle,very thick,draw,color=join,text=white,fill=join] at (3,0) {4};
\draw[->,thick]  (2) to (1);
\draw[->,thick] (3) to (2);
\draw[->,thick] (3) to (4);
\end{tikzpicture}}};
\node (c) [align=center] at (2,2) {\scalebox{0.4}{\begin{tikzpicture}
\node (1) [circle,very thick,draw,color=join,text=black] at (0,0) {1};
\node (2) [circle,very thick,draw,color=join,text=white,fill=join] at (1,0) {2};
\node (3) [circle,very thick,draw,color=meet,text=black] at (2,0) {3};
\node (4) [circle,very thick,draw,color=meet,text=black,fill=meet] at (3,0) {4};
\draw[->,thick]  (2) to (1);
\draw[->,thick] (3) to (2);
\draw[->,thick] (3) to (4);
\end{tikzpicture}}};
\node (f) [align=center] at (0,2) {\scalebox{0.4}{\begin{tikzpicture}
\node (1) [circle,very thick,draw,color=join,text=white,fill=join] at (0,0) {1};
\node (2) [circle,very thick,draw,color=meet,text=black,fill=meet] at (1,0) {2};
\node (3) [circle,very thick,draw,color=meet,text=black] at (2,0) {3};
\node (4) [circle,very thick,draw,color=join,text=white,fill=join] at (3,0) {4};
\draw[->,thick]  (2) to (1);
\draw[->,thick] (3) to (2);
\draw[->,thick] (3) to (4);
\end{tikzpicture}}};
\node (d) [align=center] at (1,1) {\scalebox{0.4}{\begin{tikzpicture}
\node (1) [circle,very thick,draw,color=join,text=white,fill=join] at (0,0) {1};
\node (2) [circle,very thick,draw,color=meet,text=black,fill=meet] at (1,0) {2};
\node (3) [circle,very thick,draw,color=meet,text=black] at (2,0) {3};
\node (4) [circle,very thick,draw,color=meet,text=black,fill=meet] at (3,0) {4};
\draw[->,thick]  (2) to (1);
\draw[->,thick] (3) to (2);
\draw[->,thick] (3) to (4);
\end{tikzpicture}}};
\node (e) [align=center] at (1,3) {\scalebox{0.4}{\begin{tikzpicture}
\node (1) [circle,very thick,draw,color=join,text=black] at (0,0) {1};
\node (2) [circle,very thick,draw,color=join,text=white,fill=join] at (1,0) {2};
\node (3) [circle,very thick,draw,color=meet,text=black,fill=meet] at (2,0) {3};
\node (4) [circle,very thick,draw,color=join,text=white,fill=join] at (3,0) {4};
\draw[->,thick]  (2) to (1);
\draw[->,thick] (3) to (2);
\draw[->,thick] (3) to (4);
\end{tikzpicture}}};
\node (g) [align=center] at (-1,3) {\scalebox{0.4}{\begin{tikzpicture}
\node (1) [circle,very thick,draw,color=meet,text=black,fill=meet] at (0,0) {1};
\node (2) [circle,very thick,draw,color=white,text=black] at (1,0) {2};
\node (3) [circle,very thick,draw,color=join,text=white,fill=join] at (2,0) {3};
\node (4) [circle,very thick,draw,color=join,text=black] at (3,0) {4};
\draw[->,thick]  (2) to (1);
\draw[->,thick] (3) to (2);
\draw[->,thick] (3) to (4);
\end{tikzpicture}}};
\node (i) [align=center] at (0,4) {\scalebox{0.4}{\begin{tikzpicture}
\node (1) [circle,very thick,draw,color=join,text=white,fill=join] at (0,0) {1};
\node (2) [circle,very thick,draw,color=join,text=black] at (1,0) {2};
\node (3) [circle,very thick,draw,color=join,text=white,fill=join] at (2,0) {3};
\node (4) [circle,very thick,draw,color=join,text=black] at (3,0) {4};
\draw[->,thick] (2) to (1);
\draw[->,thick] (3) to (2);
\draw[->,thick] (3) to (4);
\end{tikzpicture}}};
\draw[-,thick] (a) to  (b);
\draw[-,thick] (f) to (e) to  (i);
\draw[-,thick] (b) to (g);
\draw[-,thick] (b) to (f);
\draw[-,thick] (d) to (f);
\draw[-,thick] (d) to (c) to (e);
\draw[-,thick] (g) to (i);
\draw[-,thick] (a) to (d);
\end{tikzpicture}}
\caption{On the left is the extremal lattice of maximal orthogonal sets of a directed path of length 4, ordered by inclusion of first component; the sets $x_\JJ$ are indicated by a blue border, while the sets $x_\MM$ are drawn with a yellow border.  Because it is not trim, this extremal lattice is not isomorphic to the top left example in~\Cref{fig:examples} (which is not even a lattice).  On the right is the (trim) extremal lattice for a second orientation, which does coincide with the top right example in~\Cref{fig:examples} by~\Cref{thm:trim_is_ind}; maximal orthogonal pairs are indicated by the color of the border, while tight orthogonal pairs are indicated by the color of the filling.}
\label{fig:max_orth}
\end{figure}

In~\cite{TW}, we represented extremal lattices in the following way, following a construction of Markowsky~\cite{markowsky1992primes}.
Any acyclic directed graph $G$ gives rise to an extremal lattice $\LL(G)$, as follows: for $X,Y \subseteq G$ with $X \cap Y = \emptyset$, we say $(X,Y)$ is an \defn{orthogonal pair} if there is no edge from any $i\in X$ to any $k\in Y$, and we say it is a \defn{maximal orthogonal pair} if $X$ and $Y$ are maximal with that property.  Clearly, to each $Y\subseteq G$, there is at most one $X$ such that $(X,Y)$ is a maximal orthogonal pair (and dually).  Then the extremal lattice $\LL(G)$ is equivalently given by \emph{either} of
\begin{align*}
(X,Y) \leq (X',Y') &\text{ if and only if } X \subseteq X', \text{ or} \\
(X,Y) \leq (X',Y') &\text{ if and only if } Y' \subseteq Y.
\end{align*}
Furthermore, the join is computed by intersecting the second terms, while meet is given by the intersection of the first terms.  If $\x$ is an element of an extremal lattice $\LL(G)$ with corresponding maximal orthogonal pair $(X,Y)$, we write $x_\JJ=X$ and $x_\MM=Y$---that is, $x_\JJ$ corresponds to the join-irreducible elements below $x$, while $x_\MM$ corresponds to the meet-irreducible elements above $x$.  Two examples are given in~\Cref{fig:max_orth}.

Conversely, we can associate an acyclic directed grath $G(\LL)$ to any extremal lattice called its \defn{Galois graph} with the property that $\LL(G(\LL))\simeq \LL$.  We refer to \cite{TW} for further details, including Markowsky's generalization of Birkhoff's fundamental theorem of distributive lattices to extremal lattices.

\subsection{Trim lattices}
An element $x$ of a lattice $\LL$ is called \defn{left modular} if for any $y\leq z$
we have the equality
\[(y\vee x)\wedge z=y\vee(x\wedge z).\]
A lattice is called \defn{left modular} if it has a maximal chain of left modular elements.

A \defn{trim lattice} is an extremal left-modular lattice.  We have already shown that if an independence poset is a lattice, then it is extremal.  Our goal is to prove that it is actually trim.

We say that a relation $y<z$ in an extremal lattice $\LL(G)$ is \defn{overlapping} if \[y_\MM\cap z_\JJ \neq \emptyset.\]

\begin{theorem}[{\cite[Theorem 3.4]{TW}}]
\label{thm:overlapping}
An extremal lattice $\LL(G)$ is trim if and only if every relation is overlapping if and only if every cover relation is overlapping.
\end{theorem}

If a cover relation is overlapping, then it overlaps in a unique element.  We may define the \defn{downward} and \defn{upward labels} of $y \in \LL(G)$ as
\begin{align*}   \down(y)&:= \{\text{the unique element of } x_\MM \cap y_\JJ : \text{ all } x \text{ such that } x \lessdot y\} \text{ and}\\
\up(y)&:= \{\text{the unique element of } y_\MM \cap z_\JJ : \text{ all } z \text{ such that } y \lessdot z\}.
\end{align*}

 In a trim lattice $\LL(G)$, there is a unique meet-irreducible element $m_g$ with $\up(m_g)=\{g\}$, and a unique join-irreducible element $j_g$ with $\down(j_g)=\{g\}$.  We proved the following recursive properties of trim lattices in \cite{TW}, which are exactly analogous to~\Cref{lem:decomposition} and~\Cref{thm:decomposition} and will allow us to relate maximal orthogonal pairs and tight orthogonal pairs.
\begin{theorem}[{\cite[Lemma 3.10, Proposition 3.11, Proposition 3.12]{TW}}]
Let $g$ be minimal in an acyclic directed graph $G$, and write $\LL_g(G) := [\hat{0},m_g]$ and  $\LL^g(G):=[j_g,\hat{1}].$  Then
\begin{enumerate}
\item $\LL(G) = \LL_g(G) \sqcup \LL^g(G)$,
\item $\LL^g(G)\simeq \LL(G_g)$,
\item $\LL_g(G)\simeq \LL(G_g^\circ)$, and
\item an element $\x \in \LL_g(G)$ if and only if $g \in \up(\x)$.
\end{enumerate}
\label{thm:trim_recurrence}
\end{theorem}

We recall that the downward and upward labels actually associate two independent sets to each element of $\LL$.

\begin{theorem}[{\cite[Corollary 5.6]{TW}}]
For $\LL$ a trim lattice, $\down$ and $\up$ are both bijections from $\LL$ to the set of independent sets of $G(\LL)$.
\label{thm:down_ind}
\end{theorem}

We will improve this in \Cref{sec:bijections}---taking both the downward and upward labels together give a tight orthogonal pair.

\subsection{Trim Lattices to Independence Posets}

The next subsections relate trim lattices and independence posets, simultaneously generalizing the bijections between order ideals and antichains, and between Coxeter-sortable elements in a finite Coxeter group and the corresponding noncrossing partitions.

We show that $\trip(G)\simeq \LL(G)$ under certain conditions on $\trip(G)$ or on $\LL(G)$.  If $\LL(G)$ is a lattice, then $\trip(G)\simeq \LL(G)$ (\Cref{thm:trim_eq}). Similarly, if $\trip(G)$ is a lattice, then also $\trip(G)\simeq \LL(G)$ (\Cref{cor:trip_a_lattice}).  We also show that if $\trip(G)$ is a lattice, then it is a trim lattice (\Cref{thm:lattice_trim}).

{
\renewcommand{\thetheorem}{\ref{thm:trim_eq}}
\begin{theorem}
If $\LL(G)$ is trim, then $\LL(G) \simeq \trip(G)$.
\end{theorem}
\addtocounter{theorem}{-1}
}

\begin{proof}
We argue by induction on $|G|$.  Let $g$ be minimal in $G$.  By \Cref{thm:trim_recurrence}, since $\LL(G)$ is trim, $\LL(G)=\LL_g(G)\sqcup \LL^g(G)$.  Similarly, $\trip(G)=\trip_g(G) \sqcup \trip^g(G)$ by \Cref{thm:decomposition}.  By induction, we conclude the isomorphism on each of $G_g^\circ$ and $G_g$: $\LL_g(G)\simeq \LL(G_g^\circ) \simeq \trip(G_g^\circ) \simeq \trip_g(G)$ and $\LL^g(G)\simeq \LL(G_g) \simeq \trip(G_g) \simeq \trip^g(G)$.  Moreover, this induction respects the labelling of cover relations by the overlapping element (for the mops) and the element flipped (for the tops).

We have only to show now that the cover relations between $\LL_g(G)$ and $\LL^g(G)$ are the same as those for $\trip_g(G)$ and $\trip^g(G)$.  But each element in $\LL_g(G)$ and $\trip_g(G)$ has an edge up by~\Cref{thm:trim_recurrence} (4) and \Cref{lem:down} (labelled by $g$).  The element $x \in \LL_g(G)$ is paired with the unique element $x' \in \LL^g(G)$ satisfying $\down(x')=\down(x)\cup \{g\}$ by \cite[Lemma 3.15]{TW}.  It is evident from the definition of flip that the same rule describes how to pair elements in $\trip_g(G)$ with elements of $\trip^g(G)$.
\end{proof}

We now analyze what can be deduced from the fact that $\trip(G)$ is a 
lattice. To begin with, we show that, if it is a lattice, it is necessarily trim.

{
\renewcommand{\thetheorem}{\ref{thm:lattice_trim}}
\begin{theorem}
If $\trip(G)$ is a lattice then it is a trim lattice.
\end{theorem}
\addtocounter{theorem}{-1}
}
\begin{proof}
Suppose that $\trip(G)$ is a lattice.  By~\Cref{lem:lattice_implies_extremal}, it is therefore an extremal lattice.  We wish to show that every cover relation is overlapping, so that $\trip(G)$ is trim by~\Cref{thm:overlapping}.  Consider the representation of $\trip(G)$ as maximal orthogonal sets, which gives the correspondence $(\al,\br)$ with $(X,Y)$, where $X$ is the set of join-irreducible elements below $(\al,\br)$ and $Y$ is the set of meet-irreducible elements above $(\al,\br)$.  By~\Cref{lem:extremal_needs}, we therefore have that $\al \subseteq X$ and $\br \subseteq Y$.  But if $g \in \br$ and $\tog_g(\al,\br) = (\al',\br')$ is a cover with corresponding maximal orthogonal sets $(X,Y)\lessdot (X',Y')$, then $\al' \cap \br = \{g\} \subseteq X' \cap Y$ and hence the cover relation is overlapping.
\end{proof}

 The following lemma describes a situation in which it is possible to start from an orthogonal pair of independent sets and produce a tight orthogonal pair by adding elements to the two sets. It is needed for the proof of the next theorem.

\begin{lemma}\label{expand} Let $(\al',\br')$ be an orthogonal pair of independent sets, and suppose that no element of $\al'$ is below an element of $\br'$ in $G$-order. Then there exists a tight orthogonal pair of independent sets $(\al,\br)$ such that $\al\supseteq \al'$ and $\br\supseteq \br'$.\end{lemma}

\begin{proof} In a linear extension $\ell$ of $G$-order, greedily add any element to $\br'$ which can be added, subject to the condition that the set remain orthogonal to $\al'$.
Let $\br$ be the resulting set. Define $\al$ to be the independent set given by \Cref{map:zeroG} such that $(\al,\br)$ is tight orthogonal.

All that is necessary to show is that $\al$ contains $\al'$. Suppose that, as we are constructing $\al$ using \Cref{map:zeroG}, that there is some element $x$ of $\al'$ which we do not add. The reason we do not add it must be because there is some $y \to x$ such that we do add $y$ to $\al$. This element $y$ necessarily has no edge to any element of $\br$.
Since, by the hypothesis on the relative order of $\al'$ and $\br'$, there is also no edge from any element of $\br'$ to $y$, we would in fact have added $y$ into $\br$, which is a contradiction. Therefore $\al$ contains $\al'$ as desired.
\end{proof}

Note that without the assumption on the relative order of $\al'$ and $\br'$, the conclusion of this lemma is false, as demonstrated by the case of the linearly oriented path \raisebox{-0.3\height}{\scalebox{0.6}{$\begin{tikzpicture}
\node (1) [circle,thick,draw,fill=white] at (0,0) {1};
\node (2) [circle,thick,draw,fill=white] at (1,0) {2};
\node (3) [circle,thick,draw,fill=white] at (2,0) {3};
\node (4) [circle,thick,draw,fill=white] at (3,0) {4};
\draw[->,thick] (2) to (1);
\draw[->,thick] (3) to (2);
\draw[->,thick] (4) to (3);
\end{tikzpicture}$}} on four vertices: there is no tight orthogonal pair with $\al\supseteq \{2\}$ and $\br\supseteq\{4\}$ (see the top left of \Cref{fig:examples}).

Assuming that $\trip(G)$ is a lattice, and thus trim by \Cref{thm:lattice_trim}, it is of the form $\LL(H)$ for some $H$. We show that in this case $G\simeq H$.

\begin{theorem}
If $\trip(G) \simeq \LL(H)$, then $G \simeq H$.
\label{thm:galois_eq}
\end{theorem}

\begin{proof}
The join-irreducibles and meet-irreducibles of the extremal lattice $\LL(H)$ are canonically identified, by \cite[Proposition 2.7]{TW}, and are also identified with the vertices of $H$. 
The vertices of $G$ are likewise canonically identified with the join-irreducibles and the meet-irreducibles of 
$\trip(G)$; and the identification of the join-irreducibles and meet-irreducibles defined in this way is the same as the identification coming from $\LL(H)$. The vertices of $G$ and of $H$ are therefore naturally identified.

Recall from~\cite{TW} that the \defn{spine} of an extremal lattice $\LL(H)$ is the collection of elements lying on maximal-length chains; the spine is a distributive sublattice, isomorphic to the distributive lattice $J(P)$ where $P$ is the poset corresponding to $H$-order.   In $\trip(G)$ the maximal chains of the spine correspond to sequences of flips in which every vertex is flipped once. By \Cref{lem:max_chain} this can certainly be done in any linear extension of $G$-order; in principle, other orders could also be possible. This shows that $G$-order is at least as strong as $H$-order. 

If $g_1 \to g_2$ is an edge of $H$, then $g_1$ is above $g_2$ in $H$-order, so $g_1$ is above $g_2$ in $G$-order, and in particular, $g_1$ is not below $g_2$ in $G$-order.
Suppose, seeking a contradiction, that there were no edge $g_1\to g_2$ in $G$. In this case $(\{g_1\},\{g_2\})$ is an orthogonal pair, and, by \Cref{expand}, there is a top $(\al,\br)$ with $g_1 \in \al$ and $g_2 \in \br$, so $j_{g_1} \leq m_{g_2}$, contradicting the fact that there is an edge $g_1 \to g_2$ in $H$.

It follows that the edges of $H$ are a subset of the edges of $G$. Since $\LL(H)$ is trim, by \Cref{thm:trim_eq}, $\LL(H)\simeq \trip(H)$. If the edges of $H$ were a strict subset of the edges of $G$, $G$ would have more independent subsets than $H$, so $\trip(G)$ would have more vertices than $\LL(H)$, which is a contradiction. Thus $G\simeq H$.
\end{proof}

\Cref{thm:lattice_trim} and \Cref{thm:galois_eq} together imply an isomorphism of tops and mops when $\trip(G)$ is a lattice.

\begin{corollary}
If $\trip(G)$ is a lattice, then $\trip(G) \simeq \LL(G)$.
\label{cor:trip_a_lattice}
\end{corollary}

\begin{proof}
Since $\trip(G)$ is a lattice, it is a trim lattice, and hence is $\LL(H)$ for some directed acyclic graph $H$.  But now $H \simeq G$ by~\Cref{thm:galois_eq}.
\end{proof}

\subsection{Bijections}
\label{sec:bijections}

\begin{theorem}
If $\LL(G)$ is a trim lattice and $x \in \LL(G)$, then
$\phi(x) = (\down(x),\up(x))$ is a tight orthogonal pair.
\end{theorem}

\begin{proof}
 For $g \in G$, write $j_g \in \LL(G)$ for the unique join irreducible with $\down(j_g)=\{g\}$ and $m_g \in \LL(G)$ for the unique meet irreducible element with $\up(m_g)=\{g\}$.  Given a collection of join-irreducible elements $J=\{j_{g_1},\ldots,j_{g_r}\}$ and meet-irreducible elements $M=\{m_{g'_1},\ldots,m_{g'_s}\}$, then there is no edge from an element of the set $\{g_1,\ldots,g_r\}$ to an element of the set $\{g'_1,\ldots,g'_s\}$ if and only if every element of $J$ is below every element of $M$ if and only if the join of $J$ is below the meet of $M$. 

By \Cref{thm:down_ind}, every $x \in \LL(G)$ is associated to a pair of independent sets $(\down(x),\up(x))$.  By \cite[Proposition 4.1]{TW}, \[\bigvee_{g \in \down(x)} j_g = x = \bigwedge_{g \in \up(x)} m_g.\]   By the previous paragraph, there are no edges from any element of $\down(x)$ to any element of $\up(x)$, so the pair $(\down(x),\up(x))$ is orthogonal.

We now argue that $(\down(x),\up(x))$ is tight.   If we increased any element of $\down(x)$ while staying independent, this new set would correspond to $\down(x')$ of some $x' \in \LL(G)$ (by \Cref{thm:down_ind}).  This $x'$ would be strictly greater than $x$---in particular, it would no longer be below every $m_g$ for $g \in \up(x)$, and so $(\down(x),\up(x))$ would no longer be orthogonal.
\end{proof}

\begin{theorem}
If $\LL(G)$ is trim---or, equivalently, if $\trip(G)$ is a lattice---then $\phi$ is an isomorphism $\LL(G) \simeq \trip(G)$.
\label{thm:trim_is_ind}
\end{theorem}
\begin{proof}
We show $\phi$ is an isomorphism by induction on $|G|$: to this end, let $g$ be minimal in $G$.  Recall that $\LL(G)=\LL_g(G)\sqcup \LL^g(G)$ and $\trip(G)=\trip_g(G) \sqcup \trip^g(G)$.  Since $\LL_g(G)\simeq \LL(G_g^\circ)$ and $\trip_g(G) \simeq \trip(G_g^\circ)$ and $x\in \LL_g(G)$ if and only if $g \in \up(x)$ and $(\al,\br)\in\trip_g(G)$ if and only if $g \in \br$, we conclude that $\phi$ is an isomorphism from $\LL_g(G) \simeq \trip_g(G)$.  Similarly, we conclude that $\phi$ is an isomorphism from $\LL^g(G) \simeq \trip^g(G)$, since $\LL^g(G)\simeq \LL(G_g)$ and $\trip^g(G)=\trip(G_g)$.

We have only to show now that the cover relations between $\LL_g(G)$ and $\LL^g(G)$ are the same as those for $\trip_g(G)$ and $\trip^g(G)$.  But this now follows from the same argument as in~\Cref{thm:trim_eq}.
\end{proof}

The recurrence used in~\Cref{thm:trim_is_ind} can be used to show that $\phi$ has the following alternative description.  Given $(X,Y) \in \LL(G)$, construct $(\al,\br) \in \trip(G)$ as follows: $\al$ is obtained from $X$ by greedily choosing elements from $X$ in a reverse linear extension order of $G$-order so that $\al$ remains an independent set; $\br$ is obtained from $Y$ by greedily choosing elements from $Y$ in a linear extension order so that $\br$ remains an independent set.

The inverse of the map $\phi$ is constructed in the following theorem.

\begin{theorem}
 Given $(\al,\br) \in \trip(G)$, set $X=\al$, $Y=\br$, and define
 \begin{enumerate}
\item[$\theta_1$:] add to $Y$ all elements of $G \setminus X$ with no arrow \emph{from} an element of $X$.
\item[$\theta_2$:] add to $X$ all elements of $G \setminus Y$ with no arrow \emph{to} an element of $Y$.
\end{enumerate}
Then both $\theta_1 \circ \theta_2$ and $\theta_2 \circ \theta_1$ are injections $\trip(G) \hookrightarrow \LL(G)$.  
Moreover, if $\trip(G)$ is a lattice, then they coincide and are inverses to $\phi$.
\end{theorem}

\begin{proof}
It is clear that both $\theta_1\circ \theta_2$ and $\theta_2 \circ \theta_1$ have images in $\LL(G)$: $\theta_2 \circ \theta_1$, we cannot add anything to $X$, but we also cannot add anything to $Y$ because it was chosen to be maximal with respect to a subset of $X$---and similarly for $\theta_1 \circ \theta_2$.

For $\theta_2 \circ \theta_1$, we first reconstruct $\al$ from $X$ by greedily choosing elements in $X$ in a reverse linear extension order---$Y$ contains those elements that don't have an arrow from an element of $\al$, so as we scan through $X$, if we see an element without an arrow from our reconstruction of $\al$, we must add it to $\al$ since otherwise it would be in $Y$.  The set $\br$ is now determined by~\Cref{thm:trip_ind}.  A similar argument works for $\theta_1 \circ \theta_2$---we first reconstruct $\br$ from $Y$ by greedily choosing elements in $Y$ in a linear extension order, and then $\al$ is determined.

Finally, when $\trip(G)$ is a lattice, these coincide with the alternative description of $\phi$ given after \Cref{thm:trim_is_ind}.
\end{proof}

In particular, the number of tight orthogonal pairs is always less than or equal to the number of maximal orthogonal pairs.  
\Cref{thm:trim_is_ind} is illustrated in~\Cref{fig:max_orth}.

\subsection{Conditions on independence lattices}
Overlapping relations in a trim lattice allow us to give a graph-theoretic condition on $G$ for when $\trip(G)$ is a lattice.

\begin{theorem}
\label{thm:five_sets}
For $G$ an acyclic directed graph, $\trip(G)$ is a lattice if and only if $G$ has no partition $G=X_1\sqcup X_2\sqcup X_3\sqcup X_4\sqcup Z$ such that
\begin{enumerate}[(i)]
\item all the $X_i$ are non-empty (but $Z$ may be empty)
	\item every element of $X_3$ has an edge from $X_4$ and to $X_2$
    \item every element of $X_2$ has an edge to $X_1$ and from $X_3$
	\item there are no edges from $X_4$ to $X_2$, from $X_4$ to $X_1$, or from $X_3$ to $X_1$
	\item every element of $Z$ has an edge from $X_4$ and an edge to $X_1$.
    \end{enumerate}
\end{theorem}

\begin{proof}
Suppose $\LL(G)$ is not trim, so that we have a cover relation $(X,Y)\lessdot (X',Y')$ with $Y \cap X' = \emptyset$.  Let $X_4 = X$, $X_3 = X' \setminus X$, $X_2 = Y \setminus Y'$, $X_1 = Y'$, and let $Z$ contain the remaining elements of $G$.  Every element of $X_3$ has an edge from some element of $X_4$ (since otherwise $Y$ should be bigger) and an edge to some element of $X_2$ (since otherwise $X$ should be bigger).  Similarly, every element of $X_2$ has an edge to some element of $X_1$ and an edge from some element of $X_3$.  There are no edges from $X_4$ to $X_2$ or to $X_1$ (this would contradict orthogonality of $(X, Y)$) or from $X_3$ to $X_1$ (this would contradict orthogonality of $(X', Y')$). Every element of $Z$, however, has a edge from $X_4$ and a edge to $X_1$.

Conversely, suppose $\LL(G)$ is trim and consider any partition of $G$ into five sets $X_4, X_3, X_2, X_1,Z$ satisfying conditions (i), (ii), (iii) and (iv).  We will show that condition (v) is violated.  Remove all edges from $X_3$ to $X_2$, so that $(X_4 \cup X_3, X_2 \cup X_1)$ is an orthogonal pair, and can be extended to a maximal orthogonal pair $(X,Y)$.  Adding the edges from $X_3$ to $X_2$ back in, we may find maximal orthogonal pairs of the form $(X\setminus X_3,Y')$ and $(X',Y \setminus X_2)$.  These pairs are comparable, and their overlap $X'\cap Y'$ consists only of vertices of $G$ with an edge from $X_3$, an edge to $X_2$, and no edge from $X_4$ or to $X_1$.
\end{proof}

\begin{figure}[htbp]
\scalebox{0.8}{
\begin{tikzpicture}[scale=2.1]
\node (a) at (0,-1) {\scalebox{0.5}{\begin{tikzpicture}[scale=1.4]
\node (1) [circle,thick,draw,color=meet,text=black,fill=meet] at (.5,-.866) {1};
\node (2) [circle,thick,draw,color=meet,text=black,fill=white]  at (-.25,-.433013) {2};
\node (3) [circle,thick,draw,color=meet,text=black,fill=meet] at (-1,0) {3};
\node (4) [circle,thick,draw,color=meet,text=black,fill=white] at (.5,0) {4};
\node (5) [circle,thick,draw,color=meet,text=black,fill=white] at (-.25,.433013) {5};
\node (6) [circle,thick,draw,color=meet,text=black,fill=meet] at (.5,.866) {6};
\draw[->,thick] (5) to (4);
\draw[->,thick] (5) to (2);
\draw[->,thick] (6) to (4);
\draw[->,thick] (4) to (1);
\draw[->,thick] (2) to (1);
\draw[->,thick] (3) to (2);
\draw[->,thick] (4) to (2);
\draw[->,thick] (5) to (3);
\draw[->,thick] (6) to (5);
\end{tikzpicture}}};
\node (b) at (-1,1) {\scalebox{0.5}{\begin{tikzpicture}[scale=1.4]
\node (1) [circle,thick,draw,color=meet,text=black,fill=meet] at (.5,-.866) {1};
\node (2) [circle,thick,draw,color=meet,text=black,fill=white]  at (-.25,-.433013) {2};
\node (3) [circle,thick,draw,color=meet,text=black,fill=meet] at (-1,0) {3};
\node (4) [circle,thick,draw,color=white,text=black] at (.5,0) {4};
\node (5) [circle,thick,draw,color=white,text=black] at (-.25,.433013) {5};
\node (6) [circle,thick,draw,color=join,fill=join,text=white] at (.5,.866) {6};
\draw[->,thick] (5) to (4);
\draw[->,thick] (5) to (2);
\draw[->,thick] (6) to (4);
\draw[->,thick] (4) to (1);
\draw[->,thick] (2) to (1);
\draw[->,thick] (3) to (2);
\draw[->,thick] (4) to (2);
\draw[->,thick] (5) to (3);
\draw[->,thick] (6) to (5);
\end{tikzpicture}}};
\node (c) at (0,1) {\scalebox{0.5}{\begin{tikzpicture}[scale=1.4]
\node (1) [circle,thick,draw,color=join,fill=join,text=white] at (.5,-.866) {1};
\node (2) [circle,thick,draw,color=meet,text=black,fill=meet]  at (-.25,-.433013) {2};
\node (3) [circle,thick,draw,color=meet,text=black,fill=white] at (-1,0) {3};
\node (4) [circle,thick,draw,color=meet,text=black,fill=white] at (.5,0) {4};
\node (5) [circle,thick,draw,color=meet,text=black,fill=white] at (-.25,.433013) {5};
\node (6) [circle,thick,draw,color=meet,text=black,fill=meet] at (.5,.866) {6};
\draw[->,thick] (5) to (4);
\draw[->,thick] (5) to (2);
\draw[->,thick] (6) to (4);
\draw[->,thick] (4) to (1);
\draw[->,thick] (2) to (1);
\draw[->,thick] (3) to (2);
\draw[->,thick] (4) to (2);
\draw[->,thick] (5) to (3);
\draw[->,thick] (6) to (5);
\end{tikzpicture}}};
\node (d) at (1,2) {\scalebox{0.5}{\begin{tikzpicture}[scale=1.4]
\node (1) [circle,thick,draw,color=join,text=black,fill=white] at (.5,-.866) {1};
\node (2) [circle,thick,draw,color=join,fill=join,text=white]  at (-.25,-.433013) {2};
\node (3) [circle,thick,draw,color=meet,text=black,fill=meet] at (-1,0) {3};
\node (4) [circle,thick,draw,color=meet,text=black,fill=meet] at (.5,0) {4};
\node (5) [circle,thick,draw,color=meet,text=black,fill=white] at (-.25,.433013) {5};
\node (6) [circle,thick,draw,color=meet,text=black,fill=white] at (.5,.866) {6};
\draw[->,thick] (5) to (4);
\draw[->,thick] (5) to (2);
\draw[->,thick] (6) to (4);
\draw[->,thick] (4) to (1);
\draw[->,thick] (2) to (1);
\draw[->,thick] (3) to (2);
\draw[->,thick] (4) to (2);
\draw[->,thick] (5) to (3);
\draw[->,thick] (6) to (5);
\end{tikzpicture}}};
\node (e) at (3,2) {\scalebox{0.5}{\begin{tikzpicture}[scale=1.4]
\node (1) [circle,thick,draw,color=meet,text=black,fill=meet] at (.5,-.866) {1};
\node (2) [circle,thick,draw,color=white,text=black]  at (-.25,-.433013) {2};
\node (3) [circle,thick,draw,color=join,fill=join,text=white] at (-1,0) {3};
\node (4) [circle,thick,draw,color=meet,text=black,fill=white] at (.5,0) {4};
\node (5) [circle,thick,draw,color=meet,text=black,fill=meet] at (-.25,.433013) {5};
\node (6) [circle,thick,draw,color=meet,text=black,fill=white] at (.5,.866) {6};
\draw[->,thick] (5) to (4);
\draw[->,thick] (5) to (2);
\draw[->,thick] (6) to (4);
\draw[->,thick] (4) to (1);
\draw[->,thick] (2) to (1);
\draw[->,thick] (3) to (2);
\draw[->,thick] (4) to (2);
\draw[->,thick] (5) to (3);
\draw[->,thick] (6) to (5);
\end{tikzpicture}}};
\node (f) at (-1,2) {\scalebox{0.5}{\begin{tikzpicture}[scale=1.4]
\node (1) [circle,thick,draw,color=join,fill=join,text=white] at (.5,-.866) {1};
\node (2) [circle,thick,draw,color=meet,text=black,fill=meet]  at (-.25,-.433013) {2};
\node (3) [circle,thick,draw,color=meet,text=black,fill=white] at (-1,0) {3};
\node (4) [circle,thick,draw,color=white,text=black] at (.5,0) {4};
\node (5) [circle,thick,draw,color=white,text=black] at (-.25,.433013) {5};
\node (6) [circle,thick,draw,color=join,fill=join,text=white] at (.5,.866) {6};
\draw[->,thick] (5) to (4);
\draw[->,thick] (5) to (2);
\draw[->,thick] (6) to (4);
\draw[->,thick] (4) to (1);
\draw[->,thick] (2) to (1);
\draw[->,thick] (3) to (2);
\draw[->,thick] (4) to (2);
\draw[->,thick] (5) to (3);
\draw[->,thick] (6) to (5);
\end{tikzpicture}}};
\node (g) at (2,3) {\scalebox{0.5}{\begin{tikzpicture}[scale=1.4]
\node (1) [circle,thick,draw,color=join,fill=join,text=white] at (.5,-.866) {1};
\node (2) [circle,thick,draw,color=join,text=black,fill=white]  at (-.25,-.433013) {2};
\node (3) [circle,thick,draw,color=join,fill=join,text=white] at (-1,0) {3};
\node (4) [circle,thick,draw,color=meet,text=black,fill=meet] at (.5,0) {4};
\node (5) [circle,thick,draw,color=meet,text=black,fill=white] at (-.25,.433013) {5};
\node (6) [circle,thick,draw,color=meet,text=black,fill=white] at (.5,.866) {6};
\draw[->,thick] (5) to (4);
\draw[->,thick] (5) to (2);
\draw[->,thick] (6) to (4);
\draw[->,thick] (4) to (1);
\draw[->,thick] (2) to (1);
\draw[->,thick] (3) to (2);
\draw[->,thick] (4) to (2);
\draw[->,thick] (5) to (3);
\draw[->,thick] (6) to (5);
\end{tikzpicture}}};
\node (h) at (0,3) {\scalebox{0.5}{\begin{tikzpicture}[scale=1.4]
\node (1) [circle,thick,draw,color=join,text=black,fill=white] at (.5,-.866) {1};
\node (2) [circle,thick,draw,color=join,text=black,fill=white]  at (-.25,-.433013) {2};
\node (3) [circle,thick,draw,color=meet,text=black,fill=meet] at (-1,0) {3};
\node (4) [circle,thick,draw,color=join,fill=join,text=white] at (.5,0) {4};
\node (5) [circle,thick,draw,color=meet,text=black,fill=white] at (-.25,.433013) {5};
\node (6) [circle,thick,draw,color=meet,text=black,fill=meet] at (.5,.866) {6};
\draw[->,thick] (5) to (4);
\draw[->,thick] (5) to (2);
\draw[->,thick] (6) to (4);
\draw[->,thick] (4) to (1);
\draw[->,thick] (2) to (1);
\draw[->,thick] (3) to (2);
\draw[->,thick] (4) to (2);
\draw[->,thick] (5) to (3);
\draw[->,thick] (6) to (5);
\end{tikzpicture}}};
\node (i) at (-1,4) {\scalebox{0.5}{\begin{tikzpicture}[scale=1.4]
\node (1) [circle,thick,draw,color=join,text=black,fill=white] at (.5,-.866) {1};
\node (2) [circle,thick,draw,color=join,fill=join,text=white]  at (-.25,-.433013) {2};
\node (3) [circle,thick,draw,color=meet,text=black,fill=meet] at (-1,0) {3};
\node (4) [circle,thick,draw,color=join,text=black,fill=white] at (.5,0) {4};
\node (5) [circle,thick,draw,color=white,text=black] at (-.25,.433013) {5};
\node (6) [circle,thick,draw,color=join,fill=join,text=white] at (.5,.866) {6};
\draw[->,thick] (5) to (4);
\draw[->,thick] (5) to (2);
\draw[->,thick] (6) to (4);
\draw[->,thick] (4) to (1);
\draw[->,thick] (2) to (1);
\draw[->,thick] (3) to (2);
\draw[->,thick] (4) to (2);
\draw[->,thick] (5) to (3);
\draw[->,thick] (6) to (5);
\end{tikzpicture}}};
\node (j) at (1,4) {\scalebox{0.5}{\begin{tikzpicture}[scale=1.4]
\node (1) [circle,thick,draw,color=join,text=black,fill=white] at (.5,-.866) {1};
\node (2) [circle,thick,draw,color=join,text=black,fill=white]  at (-.25,-.433013) {2};
\node (3) [circle,thick,draw,color=join,fill=join,text=white] at (-1,0) {3};
\node (4) [circle,thick,draw,color=join,fill=join,text=white] at (.5,0) {4};
\node (5) [circle,thick,draw,color=meet,text=black,fill=meet] at (-.25,.433013) {5};
\node (6) [circle,thick,draw,color=meet,text=black,fill=white] at (.5,.866) {6};
\draw[->,thick] (5) to (4);
\draw[->,thick] (5) to (2);
\draw[->,thick] (6) to (4);
\draw[->,thick] (4) to (1);
\draw[->,thick] (2) to (1);
\draw[->,thick] (3) to (2);
\draw[->,thick] (4) to (2);
\draw[->,thick] (5) to (3);
\draw[->,thick] (6) to (5);
\end{tikzpicture}}};
\node (k) at (3,4) {\scalebox{0.5}{\begin{tikzpicture}[scale=1.4]
\node (1) [circle,thick,draw,color=meet,text=black,fill=meet] at (.5,-.866) {1};
\node (2) [circle,thick,draw,color=white,text=black]  at (-.25,-.433013) {2};
\node (3) [circle,thick,draw,color=join,text=black,fill=white] at (-1,0) {3};
\node (4) [circle,thick,draw,color=white,text=black] at (.5,0) {4};
\node (5) [circle,thick,draw,color=join,fill=join,text=white] at (-.25,.433013) {5};
\node (6) [circle,thick,draw,color=meet,text=black,fill=meet] at (.5,.866) {6};
\draw[->,thick] (5) to (4);
\draw[->,thick] (5) to (2);
\draw[->,thick] (6) to (4);
\draw[->,thick] (4) to (1);
\draw[->,thick] (2) to (1);
\draw[->,thick] (3) to (2);
\draw[->,thick] (4) to (2);
\draw[->,thick] (5) to (3);
\draw[->,thick] (6) to (5);
\end{tikzpicture}}};
\node (l) at (2,5) {\scalebox{0.5}{\begin{tikzpicture}[scale=1.4]
\node (1) [circle,thick,draw,color=join,fill=join,text=white] at (.5,-.866) {1};
\node (2) [circle,thick,draw,color=join,text=black,fill=white]  at (-.25,-.433013) {2};
\node (3) [circle,thick,draw,color=join,text=black,fill=white] at (-1,0) {3};
\node (4) [circle,thick,draw,color=join,text=black,fill=white] at (.5,0) {4};
\node (5) [circle,thick,draw,color=join,fill=join,text=white] at (-.25,.433013) {5};
\node (6) [circle,thick,draw,color=meet,text=black,fill=meet] at (.5,.866) {6};
\draw[->,thick] (5) to (4);
\draw[->,thick] (5) to (2);
\draw[->,thick] (6) to (4);
\draw[->,thick] (4) to (1);
\draw[->,thick] (2) to (1);
\draw[->,thick] (3) to (2);
\draw[->,thick] (4) to (2);
\draw[->,thick] (5) to (3);
\draw[->,thick] (6) to (5);
\end{tikzpicture}}};
\node (m) at (3,5) {\scalebox{0.5}{\begin{tikzpicture}[scale=1.4]
\node (1) [circle,thick,draw,color=meet,text=black,fill=meet] at (.5,-.866) {1};
\node (2) [circle,thick,draw,color=white,text=black]  at (-.25,-.433013) {2};
\node (3) [circle,thick,draw,color=join,fill=join,text=white] at (-1,0) {3};
\node (4) [circle,thick,draw,color=white,text=black] at (.5,0) {4};
\node (5) [circle,thick,draw,color=join,text=black,fill=white] at (-.25,.433013) {5};
\node (6) [circle,thick,draw,color=join,fill=join,text=white] at (.5,.866) {6};
\draw[->,thick] (5) to (4);
\draw[->,thick] (5) to (2);
\draw[->,thick] (6) to (4);
\draw[->,thick] (4) to (1);
\draw[->,thick] (2) to (1);
\draw[->,thick] (3) to (2);
\draw[->,thick] (4) to (2);
\draw[->,thick] (5) to (3);
\draw[->,thick] (6) to (5);
\end{tikzpicture}}};
\node (n) at (2,7) {\scalebox{0.5}{\begin{tikzpicture}[scale=1.4]
\node (1) [circle,thick,draw,color=join,fill=join,text=white] at (.5,-.866) {1};
\node (2) [circle,thick,draw,color=join,fill=white,text=black]  at (-.25,-.433013) {2};
\node (3) [circle,thick,draw,color=join,fill=join,text=white] at (-1,0) {3};
\node (4) [circle,thick,draw,color=join,text=black,fill=white] at (.5,0) {4};
\node (5) [circle,thick,draw,color=join,text=black,fill=white] at (-.25,.433013) {5};
\node (6) [circle,thick,draw,color=join,fill=join,text=white] at (.5,.866) {6};
\draw[->,thick] (5) to (4);
\draw[->,thick] (5) to (2);
\draw[->,thick] (6) to (4);
\draw[->,thick] (4) to (1);
\draw[->,thick] (2) to (1);
\draw[->,thick] (3) to (2);
\draw[->,thick] (4) to (2);
\draw[->,thick] (5) to (3);
\draw[->,thick] (6) to (5);
\end{tikzpicture}}};
\draw[-,thick,color=join,line width=2pt] (a) to node[midway, left] {$6$} (b);
\draw[-,thick,color=join,line width=2pt] (a) to node[midway, right] {$1$} (c) to node[midway, left] {$6$} (f);
\draw[-,thick,color=join,line width=2pt] (a) to node[midway, right] {$3$} (e) to node[midway, right] {$5$} (k) to node[midway, right] {$6$} (m);
\draw[-,thick,color=join,line width=2pt] (c) to node[midway, right] {$2$} (d) to node[midway, left] {$4$} (h) to node[midway, left] {$6$} (i);
\draw[-,thick,color=join,line width=2pt] (d) to node[midway, right] {$3$} (g) to node[midway, right] {$4$} (j) to node[midway, left] {$5$} (l) to node[midway, left] {$6$} (n);
\draw[-,thick] (b) to node[midway, left] {$1$} (f) to node[midway, left] {$2$} (i) to node[midway, left] {$3$} (n);
\draw[-,thick] (e) to node[midway, right] {$1$} (g);
\draw[-,thick] (h) to node[midway, left] {$3$} (j);
\draw[-,thick] (k) to node[midway, left] {$1$} (l);
\draw[-,thick] (b) to node[midway, left] {$3$} (m);
\draw[-,thick] (m) to node[midway, right] {$1$} (n);
\end{tikzpicture}}
\caption{The Tamari lattice with 14 elements, realized as an independence poset.  The thick blue edges indicate the tree structure provided by the natural labelling, giving an efficient method to generate all independent sets of the underlying graph.  The filling of the vertices of the graph specify the tight orthogonal pairs, while the color of the boundaries specify the maximal orthogonal pair.}
\label{fig:tamari}
\end{figure}
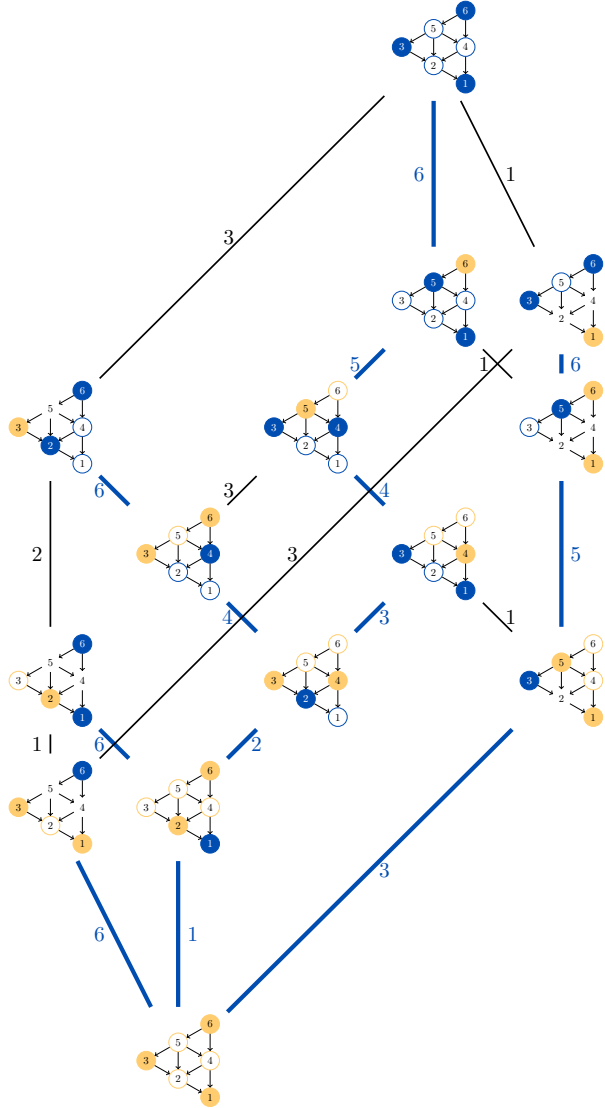

\subsection{Independence posets that are not lattices}
\label{sec:nonlattices}
In this section we highlight several differences between the behavior of independence posets that are (trim) lattices, and independence posets that are not lattices.

 Independence posets break the rigidity of the spine of a trim lattice (which we recall consists of those elements lying on maximal-length chains): although linear extensions of $G$-order still index \emph{certain} maximal chains of $\trip(G)$, there can now be other maximal chains.  In fact, the length of a maximal length chain in $\trip(G)$ can be strictly greater than $|G|$---that is, the same element can be flipped from $\br$ to $\al$ multiple times in the same chain.

\begin{example}
We leave it to the reader to check that the length of the longest chain in independence poset associated to a directed path on five vertices is not five, but six.\end{example}

Furthermore, although an independent poset that is a lattice arises from a unique Galois graph---since an extremal lattice determines its Galois graph, up to isomorphism---\Cref{prop:more_than_one} shows that uniqueness does not hold for general independence posets.

\begin{proposition}
Nonisomorphic directed acyclic graphs can give isomorphic independence posets.
\label{prop:more_than_one}
\end{proposition}

\begin{proof}
Let $G$ be as in~\Cref{fig:non_unique}---observe that $\trip(G)\simeq\trip(G)^*$, but that $G^* \not\simeq G$.  By~\Cref{prop:dual}, this gives the desired example.
\end{proof}

As a referee observed, although trim lattices are always EL-shellable with M\"obius function taking only the values $-1$, $0$, or $1$, neither of these properties are true for independence posets.  For example, the M\"obius function on the independence poset built from a linearly ordered path with six vertices attains a maximum value of 4; one can check that the independence poset built from a linearly order path with four vertices is already not shellable. 

\begin{figure}[htbp]
\raisebox{-0.5\height}{\scalebox{.9}{\begin{tikzpicture}[scale=2]
\node (a) at (0,0) {\scalebox{.5}{
\begin{tikzpicture}[scale=1]
\node (1) [circle,thick,draw,fill=meet] at (1,1) {1};
\node (2) [circle,thick,draw,fill=white]  at (0,0) {2};
\node (3) [circle,thick,draw,fill=white] at (0,2) {3};
\node (4) [circle,thick,draw,fill=meet] at (-1,1) {4};
\node (5) [circle,thick,draw,fill=white]  at (-2,1) {5};
\draw[->,thick] (3) to (1);
\draw[->,thick] (3) to (2);
\draw[->,thick] (4) to (2);
\draw[->,thick] (2) to (1);
\draw[->,thick] (4) to (3);
\draw[->,thick] (5) to (4);
\end{tikzpicture}}};
\node (b) at (-1,1) {\scalebox{.5}{
\begin{tikzpicture}[scale=1]
\node (1) [circle,thick,draw,fill=meet] at (1,1) {1};
\node (2) [circle,thick,draw,fill=white]  at (0,0) {2};
\node (3) [circle,thick,draw,fill=white] at (0,2) {3};
\node (4) [circle,thick,draw,fill=join,text=white] at (-1,1) {4};
\node (5) [circle,thick,draw,fill=meet]  at (-2,1) {5};
\draw[->,thick] (3) to (1);
\draw[->,thick] (3) to (2);
\draw[->,thick] (4) to (2);
\draw[->,thick] (2) to (1);
\draw[->,thick] (4) to (3);
\draw[->,thick] (5) to (4);
\end{tikzpicture}}};
\node (c) at (-1,2) {\scalebox{.5}{
\begin{tikzpicture}[scale=1]
\node (1) [circle,thick,draw,fill=meet] at (1,1) {1};
\node (2) [circle,thick,draw,fill=white]  at (0,0) {2};
\node (3) [circle,thick,draw,fill=white] at (0,2) {3};
\node (4) [circle,thick,draw,fill=white] at (-1,1) {4};
\node (5) [circle,thick,draw,fill=join,text=white]  at (-2,1) {5};
\draw[->,thick] (3) to (1);
\draw[->,thick] (3) to (2);
\draw[->,thick] (4) to (2);
\draw[->,thick] (2) to (1);
\draw[->,thick] (4) to (3);
\draw[->,thick] (5) to (4);
\end{tikzpicture}}};
\node (d) at (1,1) {\scalebox{.5}{
\begin{tikzpicture}[scale=1]
\node (1) [circle,thick,draw,fill=join,text=white] at (1,1) {1};
\node (2) [circle,thick,draw,fill=meet]  at (0,0) {2};
\node (3) [circle,thick,draw,fill=white] at (0,2) {3};
\node (4) [circle,thick,draw,fill=white] at (-1,1) {4};
\node (5) [circle,thick,draw,fill=meet]  at (-2,1) {5};
\draw[->,thick] (3) to (1);
\draw[->,thick] (3) to (2);
\draw[->,thick] (4) to (2);
\draw[->,thick] (2) to (1);
\draw[->,thick] (4) to (3);
\draw[->,thick] (5) to (4);
\end{tikzpicture}}};
\node (e) at (0,3) {\scalebox{.5}{
\begin{tikzpicture}[scale=1]
\node (1) [circle,thick,draw,fill=join,text=white] at (1,1) {1};
\node (2) [circle,thick,draw,fill=meet]  at (0,0) {2};
\node (3) [circle,thick,draw,fill=white] at (0,2) {3};
\node (4) [circle,thick,draw,fill=white] at (-1,1) {4};
\node (5) [circle,thick,draw,fill=join,text=white]  at (-2,1) {5};
\draw[->,thick] (3) to (1);
\draw[->,thick] (3) to (2);
\draw[->,thick] (4) to (2);
\draw[->,thick] (2) to (1);
\draw[->,thick] (4) to (3);
\draw[->,thick] (5) to (4);
\end{tikzpicture}}};
\node (f) at (2,2) {\scalebox{.5}{
\begin{tikzpicture}[scale=1]
\node (1) [circle,thick,draw,fill=white] at (1,1) {1};
\node (2) [circle,thick,draw,fill=join,text=white]  at (0,0) {2};
\node (3) [circle,thick,draw,fill=meet] at (0,2) {3};
\node (4) [circle,thick,draw,fill=white] at (-1,1) {4};
\node (5) [circle,thick,draw,fill=meet]  at (-2,1) {5};
\draw[->,thick] (3) to (1);
\draw[->,thick] (3) to (2);
\draw[->,thick] (4) to (2);
\draw[->,thick] (2) to (1);
\draw[->,thick] (4) to (3);
\draw[->,thick] (5) to (4);
\end{tikzpicture}}};
\node (g) at (3,3) {\scalebox{.5}{
\begin{tikzpicture}[scale=1]
\node (1) [circle,thick,draw,fill=white] at (1,1) {1};
\node (2) [circle,thick,draw,fill=white]  at (0,0) {2};
\node (3) [circle,thick,draw,fill=join,text=white] at (0,2) {3};
\node (4) [circle,thick,draw,fill=meet] at (-1,1) {4};
\node (5) [circle,thick,draw,fill=white]  at (-2,1) {5};
\draw[->,thick] (3) to (1);
\draw[->,thick] (3) to (2);
\draw[->,thick] (4) to (2);
\draw[->,thick] (2) to (1);
\draw[->,thick] (4) to (3);
\draw[->,thick] (5) to (4);
\end{tikzpicture}}};
\node (h) at (3,4) {\scalebox{.5}{
\begin{tikzpicture}[scale=1]
\node (1) [circle,thick,draw,fill=join,text=white] at (1,1) {1};
\node (2) [circle,thick,draw,fill=white]  at (0,0) {2};
\node (3) [circle,thick,draw,fill=white] at (0,2) {3};
\node (4) [circle,thick,draw,fill=join,text=white] at (-1,1) {4};
\node (5) [circle,thick,draw,fill=meet]  at (-2,1) {5};
\draw[->,thick] (3) to (1);
\draw[->,thick] (3) to (2);
\draw[->,thick] (4) to (2);
\draw[->,thick] (2) to (1);
\draw[->,thick] (4) to (3);
\draw[->,thick] (5) to (4);
\end{tikzpicture}}};
\node (i) at (1,4) {\scalebox{.5}{
\begin{tikzpicture}[scale=1]
\node (1) [circle,thick,draw,fill=white] at (1,1) {1};
\node (2) [circle,thick,draw,fill=join,text=white]  at (0,0) {2};
\node (3) [circle,thick,draw,fill=meet] at (0,2) {3};
\node (4) [circle,thick,draw,fill=white] at (-1,1) {4};
\node (5) [circle,thick,draw,fill=join,text=white]  at (-2,1) {5};
\draw[->,thick] (3) to (1);
\draw[->,thick] (3) to (2);
\draw[->,thick] (4) to (2);
\draw[->,thick] (2) to (1);
\draw[->,thick] (4) to (3);
\draw[->,thick] (5) to (4);
\end{tikzpicture}}};
\node (j) at (2,5) {\scalebox{.5}{
\begin{tikzpicture}[scale=1]
\node (1) [circle,thick,draw,fill=white] at (1,1) {1};
\node (2) [circle,thick,draw,fill=white]  at (0,0) {2};
\node (3) [circle,thick,draw,fill=join,text=white] at (0,2) {3};
\node (4) [circle,thick,draw,fill=white] at (-1,1) {4};
\node (5) [circle,thick,draw,fill=join,text=white]  at (-2,1) {5};
\draw[->,thick] (3) to (1);
\draw[->,thick] (3) to (2);
\draw[->,thick] (4) to (2);
\draw[->,thick] (2) to (1);
\draw[->,thick] (4) to (3);
\draw[->,thick] (5) to (4);
\end{tikzpicture}}};
\draw[-,thick] (a) to (d) to (f) to (g) to (h) to (j) to (i) to (e) to (c) to (b) to (a);
\draw[-,thick] (b) to (h);
\draw[-,thick] (d) to (e);
\draw[-,thick] (f) to (i);
\end{tikzpicture}}}
\caption{An independence poset $\trip(G)$ with the property that $\trip(G)\simeq\trip(G)^*$ but $G \not \simeq G^*$.}
\label{fig:non_unique}
\end{figure}

\section{Toggles}
\label{sec:toggles}
Fix an undirected graph $G$ and any element $g \in G$.  In~\cite[Section 3.6]{striker2016rowmotion} and in~\cite{joseph2017toggling}, a \defn{toggle} of an independent set $\mathcal{I}$ of graph $G$ is defined by~\Cref{eq:mut2}
\begin{equation}
\mut_g(\mathcal{I}) = \begin{cases} \mathcal{I} \cup \{g\} & \text{if } g \not \in \mathcal{I} \text{ and } \mathcal{I} \cup \{g\} \text{ is an independent set}, \\ \mathcal{I} \setminus \{g\} & \text{if } g \in \mathcal{I}, \\ \mathcal{I} & \text{ otherwise.} \end{cases}
\label{eq:mut2}
\end{equation}

Toggles appear naturally in independence posets through an operation similar to quiver mutation on the underlying directed graph $G$, but are defined in our context only when $g$ is an extremal element of $G$.  For $g$ an extremal element of $G$, the \defn{toggle} of the graph $G$ at $g$ is the acyclic directed graph $\mut_g(G)$ obtained by reversing all edges incident to $g$.

We have chosen the term ``toggle'' for consistency with~\cite{striker2012promotion,striker2016rowmotion}.  Although \cite{striker2016rowmotion,joseph2017toggling} define toggles as bijections on a fixed set of independent sets, our operation comes from changing the orientation of  the underlying directed graph $G$---and hence the underlying set of tops is not fixed.  We show below that we can restrict to one component of the tops to recover their bijections on independent sets.

It is clear that $\mut_g^2(G)=G$.  When $g$ is extremal, by~\Cref{lem:decomposition} no element of $\trip^g(G)$ lies below an element of $\trip_g(G)$ in $\trip(G)$.  The effect of toggling at $g$ is to reverse the roles of $\trip_g(G)$ and $\trip^g(G)$---roughly, transporting $\trip_g(G)$ above $\trip^g(G)$.  This relationship between $\trip(G)$ and $\trip(\mut_g(G))$ is summarized in~\Cref{thm:mutation_bijection}, and is illustrated in~\Cref{fig:toggle}.

\begin{theorem}
Let $g$ be a minimal element of an acyclic directed graph $G$.  Then
 \begin{itemize} \item $(\al,\br)\mapsto (\al \cup \{g\} ,\br \setminus \{g\})$ is a bijection $\trip_g(G) \simeq \trip^g(\mut_g(G))$,
 \item $(\al,\br)\mapsto (\al \setminus \{g\} ,\br')$ is a bijection $\trip^g(G) \simeq \trip_g(\mut_g(G))$.
 \end{itemize}
Let $g$ be a maximal element of an acyclic directed graph $G$.  Then
 \begin{itemize} \item $(\al,\br)\mapsto (\al',\br \setminus \{g\})$ is a bijection $\trip_g(G) \simeq \trip^g(\mut_g(G))$,
 \item $(\al,\br)\mapsto (\al \setminus \{g\} ,\br \cup \{g\})$ is a bijection $\trip^g(G) \simeq \trip_g(\mut_g(G))$.
 \end{itemize}
\label{thm:mutation_bijection}
\end{theorem}

\begin{proof}
This follows immediately from~\Cref{thm:decomposition}.
\end{proof}

For $g\in G$ extremal, we also write $\mut_g$ for the bijection of~\Cref{thm:mutation_bijection} from $\trip(G) \to \trip(\mut_g(G))$, and call it a \defn{toggle}. By~\Cref{thm:mutation_bijection}, we have that $\mut_g^2(\al,\br)=(\al,\br)$.

Thus, if we limit ourselves to toggling only at minimal elements of $G$ and keeping track of the first components $\al$, then~\Cref{eq:mut2} allows us to compute $\mut_g$ on independent sets.  \Cref{eq:mut2} also applies if we choose only maximal elements of $G$ and look at only the second components $\br$ in~\Cref{thm:mutation_bijection}.

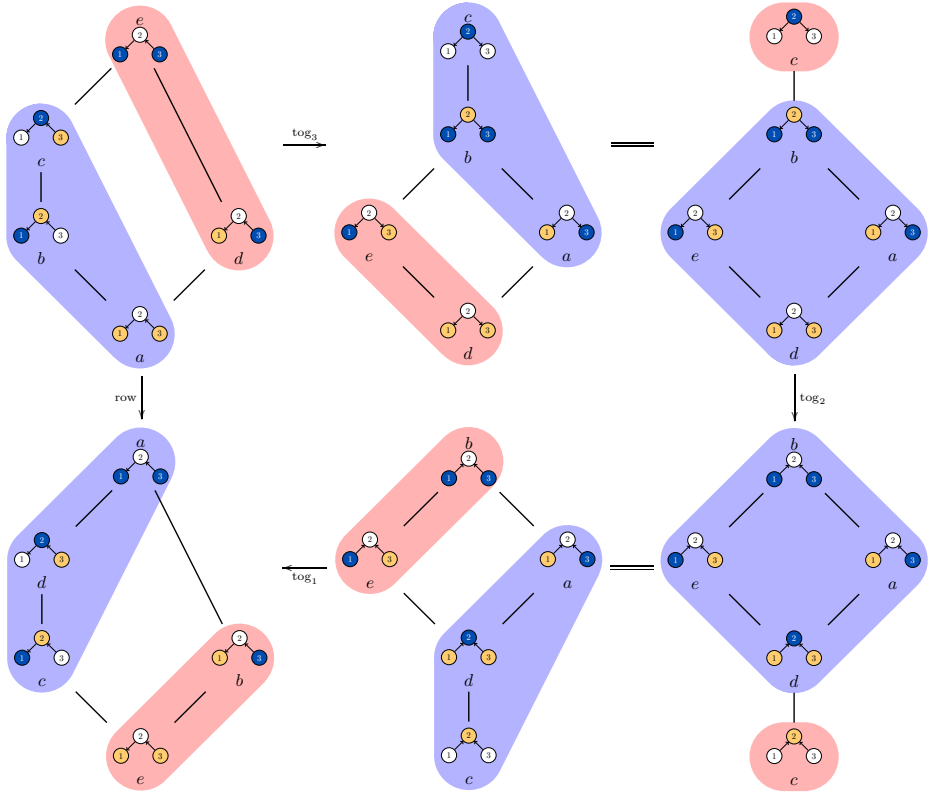
\begin{figure}[htbp]

\scalebox{.65}{
\xymatrix{
\raisebox{-0.6\height}{\begin{tikzpicture}[scale=2]
\node (x0) [align=center] at (0,0) {\scalebox{.5}{\begin{tikzpicture}[scale=.8]
\node (a) [draw=black,circle,thick,fill=meet] at (0,0) {$1$};
\node (b) [draw=black,circle,thick,fill=meet] at (2,0) {$3$};
\node (c) [draw=black,circle,thick,fill=white] at (1,1) {$2$};
\draw[<-,thick] (a) to (c); \draw[->,thick] (b) to (c); \end{tikzpicture}}\\$a$};
\node (x1) [align=center]at (-1,1) {\scalebox{.5}{\begin{tikzpicture}[scale=.8]
\node (a) [draw=black,circle,thick,fill=join,text=white] at (0,0) {$1$};
\node (b) [draw=black,circle,thick,fill=white] at (2,0) {$3$};
\node (c) [draw=black,circle,thick,fill=meet] at (1,1) {$2$};
\draw[<-,thick] (a) to (c); \draw[->,thick] (b) to (c); \end{tikzpicture}}\\$b$};
\node (j2) [align=center] at (1,1) {\scalebox{.5}{\begin{tikzpicture}[scale=.8]
\node (a) [draw=black,circle,thick,fill=meet] at (0,0) {$1$}; \node (b) [draw=black,circle,thick,fill=join,text=white] at (2,0) {$3$}; \node (c) [draw=black,circle,thick,fill=white] at (1,1) {$2$};
\draw[<-,thick] (a) to (c); \draw[->,thick] (b) to (c); \end{tikzpicture}}\\$d$};
\node (x2) [align=center] at (-1,2) {\scalebox{.5}{\begin{tikzpicture}[scale=.8]
\node (a) [draw=black,circle,thick,fill=white] at (0,0) {$1$};
\node (b) [draw=black,circle,thick,fill=meet] at (2,0) {$3$};
\node (c) [draw=black,circle,thick,fill=join,text=white] at (1,1) {$2$};
\draw[<-,thick] (a) to (c); \draw[->,thick] (b) to (c); \end{tikzpicture}} \\$c$};
\node (x3) [align=center] at (0,3) {$e$ \\\scalebox{.5}{\begin{tikzpicture}[scale=.8]
\node (a) [draw=black,circle,thick,fill=join,text=white] at (0,0) {$1$};
\node (b) [draw=black,circle,thick,fill=join,text=white] at (2,0) {$3$};
\node (c) [draw=black,circle,thick,fill=white] at (1,1) {$2$};
\draw[<-,thick] (a) to (c); \draw[->,thick] (b) to (c); \end{tikzpicture}}};
\draw[-,thick] (x0) to (x1) to (x2) to (x3);
\draw[-,thick] (x0) to (j2) to (x3);
\begin{pgfonlayer}{background}
\fill[red,opacity=0.3] \convexpath{x3,j2}{10pt};
\fill[blue,opacity=0.3] \convexpath{x2,x0,x1,x2}{10pt};
\end{pgfonlayer}
\end{tikzpicture}} \ar[r]^{\mut_3} \ar[d]_(.6){\row}
&
\raisebox{-0.6\height}{\begin{tikzpicture}[scale=2]
\node (x0) [align=center] at (0,0) {\scalebox{.5}{\begin{tikzpicture}[scale=.8]
\node (a) [draw=black,circle,thick,fill=meet] at (0,0) {$1$};
\node (b) [draw=black,circle,thick,fill=meet] at (2,0) {$3$};
\node (c) [draw=black,circle,thick,fill=white] at (1,1) {$2$};
\draw[<-,thick] (a) to (c); \draw[<-,thick] (b) to (c); \end{tikzpicture}} \\$d$};
\node (x1) [align=center] at (-1,1) {\scalebox{.5}{\begin{tikzpicture}[scale=.8]
\node (a) [draw=black,circle,thick,fill=join,text=white] at (0,0) {$1$};
\node (b) [draw=black,circle,thick,fill=meet] at (2,0) {$3$};
\node (c) [draw=black,circle,thick,fill=white] at (1,1) {$2$};
\draw[<-,thick] (a) to (c); \draw[<-,thick] (b) to (c); \end{tikzpicture}} \\$e$};
\node (j2) [align=center] at (1,1) {\scalebox{.5}{\begin{tikzpicture}[scale=.8]
\node (a) [draw=black,circle,thick,fill=meet] at (0,0) {$1$};
\node (b) [draw=black,circle,thick,fill=join,text=white] at (2,0) {$3$};
\node (c) [draw=black,circle,thick,fill=white] at (1,1) {$2$};
\draw[<-,thick] (a) to (c); \draw[<-,thick] (b) to (c); \end{tikzpicture}}\\$a$};
\node (x2) [align=center] at (0,2) {\scalebox{.5}{\begin{tikzpicture}[scale=.8]
\node (a) [draw=black,circle,thick,fill=join,text=white] at (0,0) {$1$};
\node (b) [draw=black,circle,thick,fill=join,text=white] at (2,0) {$3$};
\node (c) [draw=black,circle,thick,fill=meet] at (1,1) {$2$};
\draw[<-,thick] (a) to (c); \draw[<-,thick] (b) to (c); \end{tikzpicture}}\\$b$};
\node (x3) [align=center] at (0,3) {$c$ \\ \scalebox{.5}{\begin{tikzpicture}[scale=.8] \node (a) [draw=black,circle,thick,fill=white] at (0,0) {$1$}; \node (b) [draw=black,circle,thick,fill=white] at (2,0) {$3$}; \node (c) [draw=black,circle,thick,fill=join,text=white] at (1,1) {$2$}; \draw[<-,thick] (a) to (c); \draw[<-,thick] (b) to (c); \end{tikzpicture}}};
\draw[-,thick] (x0) to (x1) to (x2) to (x3);
\draw[-,thick] (x0) to (j2) to (x2);
\begin{pgfonlayer}{background}
\fill[red,opacity=0.3] \convexpath{x1,x0}{10pt};
\fill[blue,opacity=0.3] \convexpath{x3,j2,x2,x3}{10pt};
\end{pgfonlayer}
\end{tikzpicture}} \ar@{=}[r] &
\raisebox{-0.6\height}{\begin{tikzpicture}[scale=2]
\node (x0) [align=center] at (0,0) {\scalebox{.5}{\begin{tikzpicture}[scale=.8]
\node (a) [draw=black,circle,thick,fill=meet] at (0,0) {$1$};
\node (b) [draw=black,circle,thick,fill=meet] at (2,0) {$3$};
\node (c) [draw=black,circle,thick,fill=white] at (1,1) {$2$};
\draw[<-,thick] (a) to (c); \draw[<-,thick] (b) to (c); \end{tikzpicture}} \\$d$};
\node (x1) [align=center] at (-1,1) {\scalebox{.5}{\begin{tikzpicture}[scale=.8]
\node (a) [draw=black,circle,thick,fill=join,text=white] at (0,0) {$1$};
\node (b) [draw=black,circle,thick,fill=meet] at (2,0) {$3$};
\node (c) [draw=black,circle,thick,fill=white] at (1,1) {$2$};
\draw[<-,thick] (a) to (c); \draw[<-,thick] (b) to (c); \end{tikzpicture}} \\$e$};
\node (j2) [align=center] at (1,1) {\scalebox{.5}{\begin{tikzpicture}[scale=.8]
\node (a) [draw=black,circle,thick,fill=meet] at (0,0) {$1$};
\node (b) [draw=black,circle,thick,fill=join,text=white] at (2,0) {$3$};
\node (c) [draw=black,circle,thick,fill=white] at (1,1) {$2$};
\draw[<-,thick] (a) to (c); \draw[<-,thick] (b) to (c); \end{tikzpicture}}\\$a$};
\node (x2) [align=center] at (0,2) {\scalebox{.5}{\begin{tikzpicture}[scale=.8]
\node (a) [draw=black,circle,thick,fill=join,text=white] at (0,0) {$1$};
\node (b) [draw=black,circle,thick,fill=join,text=white] at (2,0) {$3$};
\node (c) [draw=black,circle,thick,fill=meet] at (1,1) {$2$};
\draw[<-,thick] (a) to (c); \draw[<-,thick] (b) to (c); \end{tikzpicture}}\\$b$};
\node (x3) [align=center] at (0,3) {\scalebox{.5}{\begin{tikzpicture}[scale=.8] \node (a) [draw=black,circle,thick,fill=white] at (0,0) {$1$}; \node (b) [draw=black,circle,thick,fill=white] at (2,0) {$3$}; \node (c) [draw=black,circle,thick,fill=join,text=white] at (1,1) {$2$}; \draw[<-,thick] (a) to (c); \draw[<-,thick] (b) to (c); \end{tikzpicture}}\\$c$};
\draw[-,thick] (x0) to  (x1) to (x2) to (x3);
\draw[-,thick] (x0) to (j2) to (x2);
\node (cd) at (-.1,3) {};
\node (cc) at (.1,3) {};
\begin{pgfonlayer}{background}
\fill[red,opacity=0.3] \convexpath{cd,cc}{10pt};
\fill[blue,opacity=0.3] \convexpath{x0,x1,x2,j2,x0}{10pt};
\end{pgfonlayer}
\end{tikzpicture}} \ar[d]^(.6){\mut_2}
\\
\raisebox{-0.6\height}{\begin{tikzpicture}[scale=2]
\node (x0) [align=center] at (0,0) {\scalebox{.5}{\begin{tikzpicture}[scale=.8]
\node (a) [draw=black,circle,thick,fill=meet] at (0,0) {$1$};
\node (b) [draw=black,circle,thick,fill=meet] at (2,0) {$3$};
\node (c) [draw=black,circle,thick,fill=white] at (1,1) {$2$};
\draw[<-,thick] (a) to (c); \draw[->,thick] (b) to (c); \end{tikzpicture}}\\$e$};
\node (x1) [align=center]at (-1,1) {\scalebox{.5}{\begin{tikzpicture}[scale=.8]
\node (a) [draw=black,circle,thick,fill=join,text=white] at (0,0) {$1$};
\node (b) [draw=black,circle,thick,fill=white] at (2,0) {$3$};
\node (c) [draw=black,circle,thick,fill=meet] at (1,1) {$2$};
\draw[<-,thick] (a) to (c); \draw[->,thick] (b) to (c); \end{tikzpicture}}\\$c$};
\node (j2) [align=center] at (1,1) {\scalebox{.5}{\begin{tikzpicture}[scale=.8]
\node (a) [draw=black,circle,thick,fill=meet] at (0,0) {$1$}; \node (b) [draw=black,circle,thick,fill=join,text=white] at (2,0) {$3$}; \node (c) [draw=black,circle,thick,fill=white] at (1,1) {$2$};
\draw[<-,thick] (a) to (c); \draw[->,thick] (b) to (c); \end{tikzpicture}}\\$b$};
\node (x2) [align=center] at (-1,2) {\scalebox{.5}{\begin{tikzpicture}[scale=.8]
\node (a) [draw=black,circle,thick,fill=white] at (0,0) {$1$};
\node (b) [draw=black,circle,thick,fill=meet] at (2,0) {$3$};
\node (c) [draw=black,circle,thick,fill=join,text=white] at (1,1) {$2$};
\draw[<-,thick] (a) to (c); \draw[->,thick] (b) to (c); \end{tikzpicture}}\\$d$ };
\node (x3) [align=center] at (0,3) {$a$\\ \scalebox{.5}{\begin{tikzpicture}[scale=.8]
\node (a) [draw=black,circle,thick,fill=join,text=white] at (0,0) {$1$};
\node (b) [draw=black,circle,thick,fill=join,text=white] at (2,0) {$3$};
\node (c) [draw=black,circle,thick,fill=white] at (1,1) {$2$};
\draw[<-,thick] (a) to (c); \draw[->,thick] (b) to (c); \end{tikzpicture}}};
\draw[-,thick] (x0) to (x1) to (x2) to (x3);
\draw[-,thick] (x0) to (j2) to (x3);
\begin{pgfonlayer}{background}
\fill[blue,opacity=0.3] \convexpath{x1,x2,x3}{10pt};
\fill[red,opacity=0.3] \convexpath{x0,j2}{10pt};
\end{pgfonlayer}
\end{tikzpicture}}
&
\raisebox{-0.6\height}{\begin{tikzpicture}[scale=2]
\node (x0) [align=center] at (0,0) {\scalebox{.5}{\begin{tikzpicture}[scale=.8]
\node (a) [draw=black,circle,thick,fill=meet] at (0,0) {$1$};
\node (b) [draw=black,circle,thick,fill=meet] at (2,0) {$3$};
\node (c) [draw=black,circle,thick,fill=join,text=white] at (1,1) {$2$};
\draw[<-,thick] (c) to (a); \draw[<-,thick] (c) to (b); \end{tikzpicture}}\\$d$};
\node (x1) [align=center] at (-1,1) {\scalebox{.5}{\begin{tikzpicture}[scale=.8]
\node (a) [draw=black,circle,thick,fill=join,text=white] at (0,0) {$1$};
\node (b) [draw=black,circle,thick,fill=meet] at (2,0) {$3$};
\node (c) [draw=black,circle,thick,fill=white] at (1,1) {$2$};
\draw[<-,thick] (c) to (a); \draw[<-,thick] (c) to (b); \end{tikzpicture}}\\$e$};
\node (j2) [align=center] at (1,1) {\scalebox{.5}{\begin{tikzpicture}[scale=.8]
\node (a) [draw=black,circle,thick,fill=meet] at (0,0) {$1$};
\node (b) [draw=black,circle,thick,fill=join,text=white] at (2,0) {$3$};
\node (c) [draw=black,circle,thick,fill=white] at (1,1) {$2$};
\draw[<-,thick] (c) to (a); \draw[<-,thick] (c) to (b); \end{tikzpicture}}\\$a$};
\node (x2) [align=center] at (0,2) {$b$ \\ \scalebox{.5}{\begin{tikzpicture}[scale=.8]
\node (a) [draw=black,circle,thick,fill=join,text=white] at (0,0) {$1$};
\node (b) [draw=black,circle,thick,fill=join,text=white] at (2,0) {$3$};
\node (c) [draw=black,circle,thick,fill=white] at (1,1) {$2$};
\draw[<-,thick] (c) to (a); \draw[<-,thick] (c) to (b); \end{tikzpicture}}};
\node (x3) [align=center] at (0,-1) {\scalebox{.5}{\begin{tikzpicture}[scale=.8]
\node (a) [draw=black,circle,thick,fill=white] at (0,0) {$1$};
\node (b) [draw=black,circle,thick,fill=white] at (2,0) {$3$};
\node (c) [draw=black,circle,thick,fill=meet] at (1,1) {$2$};
\draw[<-,thick] (c) to (a); \draw[<-,thick] (c) to (b); \end{tikzpicture}}\\$c$};
\draw[-,thick] (x0) to (x1) to  (x2);
\draw[-,thick] (x0) to (j2) to (x2);
\draw[-,thick] (x3) to (x0);
\begin{pgfonlayer}{background}
\fill[blue,opacity=0.3] \convexpath{x3,x0,j2}{10pt};
\fill[red,opacity=0.3] \convexpath{x1,x2}{10pt};
\end{pgfonlayer}
\end{tikzpicture}} \ar[l]^{\mut_1} \ar@{=}[r]
& \raisebox{-0.6\height}{\begin{tikzpicture}[scale=2]
\node (x0) [align=center] at (0,0) {\scalebox{.5}{\begin{tikzpicture}[scale=.8]
\node (a) [draw=black,circle,thick,fill=meet] at (0,0) {$1$};
\node (b) [draw=black,circle,thick,fill=meet] at (2,0) {$3$};
\node (c) [draw=black,circle,thick,fill=join,text=white] at (1,1) {$2$};
\draw[<-,thick] (c) to (a); \draw[<-,thick] (c) to (b); \end{tikzpicture}}\\$d$};
\node (x1) [align=center] at (-1,1) {\scalebox{.5}{\begin{tikzpicture}[scale=.8]
\node (a) [draw=black,circle,thick,fill=join,text=white] at (0,0) {$1$};
\node (b) [draw=black,circle,thick,fill=meet] at (2,0) {$3$};
\node (c) [draw=black,circle,thick,fill=white] at (1,1) {$2$};
\draw[<-,thick] (c) to (a); \draw[<-,thick] (c) to (b); \end{tikzpicture}}\\$e$};
\node (j2) [align=center] at (1,1) {\scalebox{.5}{\begin{tikzpicture}[scale=.8]
\node (a) [draw=black,circle,thick,fill=meet] at (0,0) {$1$};
\node (b) [draw=black,circle,thick,fill=join,text=white] at (2,0) {$3$};
\node (c) [draw=black,circle,thick,fill=white] at (1,1) {$2$};
\draw[<-,thick] (c) to (a); \draw[<-,thick] (c) to (b); \end{tikzpicture}}\\$a$};
\node (x2) [align=center] at (0,2) {$b$ \\ \scalebox{.5}{\begin{tikzpicture}[scale=.8]
\node (a) [draw=black,circle,thick,fill=join,text=white] at (0,0) {$1$};
\node (b) [draw=black,circle,thick,fill=join,text=white] at (2,0) {$3$};
\node (c) [draw=black,circle,thick,fill=white] at (1,1) {$2$};
\draw[<-,thick] (c) to (a); \draw[<-,thick] (c) to (b); \end{tikzpicture}}};
\node (x3) [align=center] at (0,-1) {\scalebox{.5}{\begin{tikzpicture}[scale=.8]
\node (a) [draw=black,circle,thick,fill=white] at (0,0) {$1$};
\node (b) [draw=black,circle,thick,fill=white] at (2,0) {$3$};
\node (c) [draw=black,circle,thick,fill=meet] at (1,1) {$2$};
\draw[<-,thick] (c) to (a); \draw[<-,thick] (c) to (b); \end{tikzpicture}}\\$c$};
\draw[-,thick] (x0) to (x1) to  (x2);
\draw[-,thick] (x0) to (j2) to (x2);
\draw[-,thick] (x3) to (x0);
\node (cd) at (-.1,-1) {};
\node (cc) at (.1,-1) {};
\begin{pgfonlayer}{background}
\fill[red,opacity=0.3] \convexpath{cd,cc}{10pt};
\fill[blue,opacity=0.3] \convexpath{x0,x1,x2,j2,x0}{10pt};
\end{pgfonlayer}
\end{tikzpicture}}
}}
\caption{Each poset is an independence poset on an orientation of a path of length $3$.  Taken together, these posets represent a sequence of toggles.  Each toggle induces a bijection between $\trip(G)$ and $\trip(\mut_g(G))$, which we keep track of by assigning letters to the elements of the posets.  Toggling each element of the directed graph underlying the independence poset at the top left (in reverse linear extension order) recovers the same indpendence poset on the bottom left, permuting its elements as $(a,b,c,d,e) \xmapsto{\mut_{321}}(e,d,b,c,a)$.  This coincides with rowmotion.}
\label{fig:toggle}
\end{figure}

\section{Rowmotion on Independence Posets}
\label{sec:rowmotion}
By~\Cref{thm:trip_ind}, any independent set $\mathcal{I}$ can be completed to a tight orthogonal pair in exactly two ways---the first as $(\mathcal{I},\br)$ and the second as $(\al,\mathcal{I})$.  \defn{Rowmotion} sends the first of these to the second (see \Cref{eq:global_row}).  The purpose of this section is to give two additional ways to compute it: one using flips, and one using toggles.

\begin{definition}
For $G$ an acyclic directed graph, we say that rowmotion on $\trip(G)$ can be computed
\begin{itemize}
\item \defn{in slow motion} if $\row = \prod\limits_{g \in \ell} \tog_{g}$, and
\item \defn{by deformotion} if $\row = \prod\limits_{g \in \ell'} \mut_{g}$,
\end{itemize}
where $\ell$ is any linear extension of $G$-order and  $\ell'$ is any reverse linear extension of $G$-order.
\end{definition}

{
\renewcommand{\thetheorem}{\ref{thm:main_thm}(i)}
\begin{theorem} For $G$ an acyclic directed graph, rowmotion on $\trip(G)$ can be computed in slow motion.
\end{theorem}
\addtocounter{theorem}{-1}
}

\begin{proof}
We follow the same proof as~\cite[Theorem 1.1]{TW}.  Let $g$ be minimal in $G$.

    \emph{Case I: $(\al,\br) \in \trip_g(G)$.} At the first step when calculating $\prod_{h \in \ell} \tog_{h}$, we walk to $(\al',\br')=\tog_g(\al,\br)$.  So $\al'=\al \cup \{g\}$ by~\Cref{lem:down} and $(\al',\br') \in \trip^g(G)$ by~\Cref{lem:decomposition}.  Applying the rest of the flips to $(\al',\br')$ has the effect of applying $\prod_{\substack{h \in \ell\\ h\neq g}} \tog_{h}$ in $\trip(G_g)$.  (Note that the only flips in $\trip(G)$ leaving $\trip^g(G)$ are those of the form $\tog_g$, so they will never be taken.)  By induction, we obtain an element $(\al'',\br'') \in \trip(G_g)$ such that $\br''=\al'\setminus \{g\}$.  Passing back to $\trip^g(G)$ does not change $\br''$, and we conclude the result.

    \emph{Case II:  $(\al,\br) \in \trip^g(G)$ and $g \in \al$.}   At the first step when calculating $\prod_{h \in \ell} \tog_{h}$, we walk to $(\al',\br')=\tog_g(\al,\br)$.  Since $g \in \br'$, $(\al',\br')\in \trip_g(G)$ by~\Cref{lem:decomposition} and $\al' \cup \{g\}=\al$ by~\Cref{lem:down}.  Applying the rest of the flips to $(\al',\br')$ has the effect of applying  $\prod_{\substack{h \in \ell\\ h\neq g}} \tog_{h}$ in $\trip_g(G)$.  (Note that the only flips in $\trip(G)$ leaving $\trip_g(G)$ are those of the form $\tog_g$, so they will never be taken.)  By induction, we obtain an element $(\al'',\br'') \in \trip_g(G)$ such that $\br''= \al' \cup\{g\}$.  But $\al'\cup\{g\}=\al$.

\emph{Case III: $(\al,\br) \in \trip^g(G)$ and $g \not \in \al$.}
By assumption, flipping at $g$ has no effect.  As in Case I, the rest of the flips have the effect of applying $\prod_{\substack{h \in \ell\\ h\neq g}} \tog_{h}$ in $\trip^g(G)$.  By induction, we obtain an element $(\al',\br') \in \trip^g(G)$ such that $\br'=\al$.
\end{proof}

{
\renewcommand{\thetheorem}{\ref{thm:main_thm}(ii)}
\begin{theorem}
For $G$ an acyclic directed graph, rowmotion on $\trip(G)$ can be computed by deformotion.
\label{thm:row_mutation}
\end{theorem}
\addtocounter{theorem}{-1}
}

\begin{proof}
Since every element of $G$ is toggled, every edge is flipped twice, and so $\prod_{g \in \ell'}\mut_g(G)=G$. By definition, $\mut_g$ takes all elements with $g \in \al$ and converts them to elements with $g \in \br$.  Since we are toggling in $\ell'=g_1,\ldots,g_{|G|}$ order, at any step we have moved all $\{g_1,\ldots,g_i\}\cap \al$ from $\al$ to $\br$.  Although this introduces some new elements into $\al$, these only involve $\{g_1,\ldots,g_i\}$ and so are never moved to $\br$ in a subsequent step.  Thus, each element has its \emph{original} set $\al$ converted to its new $\br$, which is the definition of rowmotion given in~\Cref{eq:global_row}.
\end{proof}

\section{Representation Theory}
\label{sec:rep_theory}

In this section, we show that the combinatorics of independence lattices arises naturally in representation theory.  In short, the torsion/torsion-free pairs of certain acyclic finite-dimensional algebras correspond to maximal orthogonal pairs, while their 2-term simpleminded collections recover tight orthogonal pairs.   The setting in which we work, while special from the point of
view of representation theory, includes many interesting examples, such as all
quotients of Dynkin path algebras.

\subsection{Representation-finite directed algebras}
Let $k$ be a field, and $A$ a finite-dimensional $k$-algebra.  Suppose further
that the module category $\mod A$ is directed, i.e., there is no sequence of pairwise non-isomorphic indecomposable modules $M_1,\dots , M_r$ for $r>1$ with non-zero morphisms from $M_i$ to $M_{i+1}$ and from $M_r$ to $M_1$.  Under this hypothesis, all indecomposable modules have some properties which
usually only belong to a subset of indecomposables.

A module is called a
\defn{brick} if its endomorphism ring is a division algebra.  Note that a brick is
necessarily indecomposable.  A module $M$ is called \defn{$\tau$-rigid} if
\[\Hom(M,\tau M)=0,\] where $\tau$ is the Auslander-Reiten translation.

\begin{proposition}\label{rigid-brick}
Let $A$ be a finite-dimensional $k$-algebra such that $\mod A$ has no cycles.  Then all indecomposable $A$-modules are $\tau$-rigid bricks.
\end{proposition}

\begin{proof}
Suppose that
$M$ is an indecomposable module and that $\phi$ is a non-invertible endomorphism
of $M$.  Let $N$ be the image of $\phi$, which is necessarily a proper submodule
of $M$ and a proper quotient of $M$.  This implies that there are
non-zero morphisms in both directions between $M$ and $N$, contrary to our
assumption.  Therefore all indecomposable modules are bricks.

If $\tau M$ is non-zero, there is a
short exact sequence \[0\rightarrow \tau M \rightarrow E \rightarrow M
\rightarrow 0.\] Thus, if $\Hom(M,\tau M)\ne 0$, then there is a cycle in $\mod A$.
Thus, all indecomposable objects are also $\tau$-rigid.
\end{proof}

We specialize further by imposing a finiteness assumption.
We assume that $A$ is representation-finite, i.e., it has only finitely
many indecomposable representations up to isomorphism.  By~\Cref{rigid-brick}, we could just as well have assumed only that the number of
$\tau$-rigid indecomposable modules is finite.

\subsection{Torsion classes}
A \defn{torsion class} in $\mod A$ is a full additive
subcategory of $\mod A$ closed under
extensions and quotients.  We write $\tors A$ for the torsion classes in $\mod A$.  Torsion classes naturally form a lattice, since the intersection of two torsion classes is again a torsion class.  Since $A$ is representation-finite, there are only
finitely many different torsion classes, so $\tors A$ is a finite lattice.
We already showed in \cite{TW} that this lattice is trim.

There is a dual notion to torsion classes, the torsion-free classes.  A
\defn{torsion-free class} is a full additive subcategory closed under extensions
and submodules.  There is a natural inclusion-reversing correspondence between
torsion classes and torsion-free classes: if $\mathcal T$ is a torsion class,
then
\[\mathcal F=\mathcal T^\perp=\{Y \mid \Hom(X,Y)=0 \textrm{ for all } X\in \mathcal{T}\}\]
is the corresponding torsion-free class.
The set of indecomposable objects of $\mathcal F$ is maximal with
respect to the property of having no morphisms from a module in
$\mathcal T$ (or equivalently, from an indecomposable module in
$\mathcal T$).  The dual statement is also true: the set of
indecomposable modules in $\mathcal T$ is maximal
with respect to the property of having no nonzero morphisms into a module
in $\mathcal F$ (or equivalently, into an indecomposable module in $\mathcal F$).  Further, every pair $(\mathcal T,\mathcal F)$ which is
maximal in both these respects is automatically a \defn{torsion pair}---a torsion class and its
corresponding torsion-free class.  See \cite[Section VI.1]{ASS}.

In this way, we see the connection with maximal orthogonal pairs.  Consider the graph $G$ whose vertices
are indexed by the indecomposable $A$-modules, with an arrow from $M$ to
$N$ if and only if $\Hom(M,N)\ne 0$.  By our assumption that there are no cycles in
$\mod A$, the graph $G$ is acyclic.  It is now clear, as we observed in~\cite{TW},
that maximal orthogonal pairs of $G$ correspond to torsion pairs.

\begin{theorem}[{\cite[Corollary 1.5]{TW}}]
If $A$ is representation finite and $\mod A$ has no cycles, then the maximal orthogonal pairs in the trim lattice $\LL(G)$ are naturally the torsion pairs of $A$, ordered with respect to inclusion of torsion classes (or reverse inclusion of torsion-free classes).
\end{theorem}

\subsection{Semibricks, flips, and the edge-labelling of the Hasse diagram of the lattice of torsion classes}

There is a natural labelling of each edge of the Hasse diagram of the lattice of torsion classes of $\mod A$ by an indecomposable module---if $\mathcal U\lessdot \mathcal V$ is a
cover relation in the lattice of torsion classes, then its label is the unique
brick in $\mathcal V\cap \mathcal U^\perp$~\cite[Proposition 1.17]{asai2016semibricks}.  Comparing with our~\Cref{thm:overlapping}, we conclude that the labelling of cover relations by bricks coincides with the labelling coming from the overlapping cover relations of mops.

A collection of bricks is called a \defn{semibrick} if there are no morphisms between two
non-isomorphic bricks in the collection.  Asai showed that there is a bijection
between torsion classes and semibricks~\cite[Theorem 1.3]{asai2016semibricks}.  The semibrick $s(\mathcal T)$ corresponding to
a torsion class $\mathcal T$ is the collection of labels on edges down from
$\mathcal T$ \cite[Lemma 1.16, Proposition 1.17]{asai2016semibricks}.  It can also be described as the unique semibrick such that $\mathcal T$
consists of modules filtered by quotients of elements of $s(T)$~\cite[Lemma 1.5]{asai2016semibricks}.  Dually, there is a semibrick corresponding to each torsion free class: the
semibrick $s(\mathcal T^\perp)$ corresponding to the torsion free class $\mathcal T^\perp$ is
the collection of labels on the edges up from $\mathcal T$.  This semibrick is
the unique semibrick such that $\mathcal T^\perp$ is filtered by submodules of
the elements of $s(\mathcal T^\perp)$.

An analogue of~\Cref{thm:trim_is_ind} allows us to directly compute the semibrick corresponding to a given torsion class.

\begin{proposition}\label{greedy} The semibrick associated to a torsion class $\mathcal T$ can
be obtained by greedily building a semibrick by adding modules from $\mathcal T$ if possible in any linear extension of the $G$-order. \end{proposition}

\begin{proof}  Let $\mathcal D=\{X_1,\dots,X_r\}$ be the semibrick associated to
$\mathcal T$.  As we consider modules from $\mathcal T$ in some linear order compatible with $G$, suppose that we have already
added $\{X_1,\dots,X_i\}$ and we encounter some module $Y$ which does not have any
morphisms from $X_1,\dots,X_i$.  Since $Y$ is in $\mathcal T$, it is filtered
by quotients of modules in $\mathcal D$.  Now $Y$ does not have any morphisms from
the already-chosen elements of $\mathcal D$, but $Y$ will not admit any
morphisms from any subsequent module.  Thus $Y$ must be an element of
$\mathcal D$ as well.  On the other hand, if $Y$ admits a
morphism from some module which we have already added to $\mathcal D$, then obviously we must not add it into $\mathcal D$, and following our procedure, we do not.
We will therefore successively add all the elements of $\mathcal D$ by following this procedure.
\end{proof}

Using semibricks, we can provide a representation-theoretic justification for our definition of flips on tight orthogonal pairs in~\Cref{map:flip}.

\begin{theorem} Let $(\mathcal D,\mathcal U)$ be the tight orthogonal
pair associated to the torsion class $\mathcal T$.  Let $X\in \mathcal U$.  Let $(\mathcal D',\mathcal U')$ be the
tight orthogonal pair associated to the torsion class  $\mathcal T'$
which covers $\mathcal T$ along an edge of the Hasse diagram labelled by $X$.  Then~\Cref{map:flip} successfully reconstructs
$(\mathcal D', \mathcal U')$.
\end{theorem}

\begin{proof}
The torsion class $\mathcal T'$ is the minimal torsion
class containing $X$ and $\mathcal T$.  Clearly it contains all modules
which are filtered by quotients of $X$ and elements of $\mathcal T$,
and since that category is extension closed and quotient closed, it must be
$\mathcal T'$.  In particular, we observe that when restricted to
modules that do not follow $X$ in $G$-order, the
two classes $\mathcal T'$ and $\mathcal T$ coincide, which we write
as $\mathcal T'|_{\not\geq X}=\mathcal T|_{\not\geq X}$.  Therefore,
by \Cref{greedy}, $\mathcal D|_{\not \geq X}=\mathcal D'|_{\not\geq X}$. Now
\[\mathcal F'=\mathcal T'^\perp=\{Y\in \mathcal F\mid \Hom(X,Y)=0\}.\]
Note that $\mathcal U|_{\not\leq X} \subseteq \mathcal F'$, since if we reconstruct $\mathcal U$ using (the dual version of) \Cref{greedy}, we must add in $X$,
and therefore we must not have added any module which admits a map
from $X$ previously.  Now, since $\mathcal U|_{\not\leq X}$ is contained in $\mathcal F'\subset \mathcal F$, we see that $\mathcal U'|_{\not\leq X}=\mathcal U|_{\not\leq X}$.

The remaining elements of $\mathcal D'$, i.e., those which follow
$X$ in $G$-order, can then be reconstructed
by \Cref{greedy}.  We first observe that $X$ itself is in $\mathcal D'$,
since it is in $\mathcal T'$, and since $X\in \mathcal U$, there are
no morphisms from $\mathcal D'|_{\not\leq X}$ to $X$.  To use
\Cref{greedy} to determine the further elements of $\mathcal D'$,
we first need to know the
elements of $\mathcal T'|_{>X}$. These are the modules strictly following
$X$ which do not admit a morphism into $\mathcal F'$.  We can replace
$\mathcal F'$ by $\mathcal U'$, and then we can replace $\mathcal U'$
by $\mathcal U'|_{>X}$, since modules following $X$ only have morphisms
into other modules following $X$.  But we have already observed that
$\mathcal U'|_{>X}=\mathcal U|_{> X}$.

We therefore see that the modules of $\mathcal D'$ which
strictly follow $X$ are constructed greedily in $G$-order
maintaining the properties of having no morphisms from any of the other modules in $\mathcal D'$, and having no morphisms to
$\mathcal U$.

This construction of $\mathcal D'$ is manifestly the same as the
construction given in \Cref{map:flip}, as desired.  The argument
for the construction of $\mathcal U'$ is essentially dual.
\end{proof}

\subsection{2-simple-minded collections}\label{sec:2simple}
A \defn{simple-minded collection} for $A$
is a collection of objects $X_1,\dots, X_r$
in the derived category of $D^b(A)$ such that:
\begin{enumerate}
\item $\Hom(X_i,X_j[m])=0$ for $m<0$,
\item $\End(X_i)$ is a division algebra and $\Hom(X_i,X_j)=0$ unless
  $i=j$,
\item $X_1,\dots, X_r$ generate $D^b(A)$ in the sense that the smallest thick subcategory containing all of them is
$D^b(A)$ itself,
\end{enumerate}
See \cite{KY} for more on simple-minded collections.  A \defn{2-simple-minded
collection} has the additional property that for each $X_i$, we have
$H^j(X_i)=0$ for $j\ne 0, -1$.  It turns out that the elements of a
2-simple-minded collection are all contained in $\mod A \cup \mod A[1]$,
see \cite[Remark 4.11]{BY}.

Asai showed that there is a bijection from torsion classes to 2-simple-minded
collections \cite[Theorem 2.3]{asai2016semibricks}, which sends
$\mathcal T$ to $s(\mathcal T) \cup s(\mathcal T^\perp)[1]$.

It follows from our bijection between torsion classes and mops, and the labelling of the cover relations by bricks and~\Cref{thm:overlapping}, that tight
orthogonal pairs are in bijection with 2-simple-minded collections.

\begin{theorem} Let $A$ be a representation finite $k$-algebra with no cycles
  in $\mod A$.  There is a bijection from tight orthogonal pairs to
  2-simple-minded collections, sending $(\mathcal D,\mathcal U)$ to
  $\mathcal D \cup \mathcal U[1]$.
\label{thm:asai}
\end{theorem}

Although~\Cref{thm:asai} follows from our results
together with those of \cite{asai2016semibricks}, we find it
instructive to give a
direct proof of the following proposition in order to show how the tightness condition  on tops naturally arises from 2-simple-minded collections.

\begin{proposition} If $\mathcal D \cup \mathcal U[1]$ is a
2-simple-minded collection, then $(\mathcal D, \mathcal U)$ is a tight orthogonal pair.
\end{proposition}

\begin{proof}
Let $\mathcal C=\{X_1,\dots,X_r,Y_1[1],\dots,Y_s[1]\}$ be a
  2-term simple minded collection.  From the condition (2), we see that
  $\mathcal D=\{X_1,\dots,X_r\}$ and
  $\mathcal U=\{Y_1,\dots,Y_s\}$ are independent sets.  The
    fact that $\Ext^{-1}(X_i,Y_i[1])=0$ implies that $\Hom(X_i,Y_j)=0$, i.e., that
  $(\mathcal D,\mathcal U)$ is orthogonal.

    The proof of tightness is somewhat more involved.  Write $\mathcal T$
    for the torsion class associated to $\mathcal D$, and $\mathcal T^\perp$ for
    the torsion-free class associated to $\mathcal U$.  If $Z$ is a
    module in
    $\mathcal T$, then it cannot be added into $\mathcal D$ because such a
    module is
    filtered by quotients of objects from $\mathcal D$, and thus admits a
    morphism from one of them, which is forbidden.
    If $Z$ is a module not in $\mathcal T$,
    then it admits a map to a non-zero object in $\mathcal T^\perp$, and
    thus to an object in $\mathcal U$; this is also forbidden.

    It is also impossible to replace $X_i$ by $X_i'$ which is above
    $X_i$.  If $X_i'$ is not in $\mathcal T$, then this is impossible for
    the reason given above that we cannot add $X_i'$ into $\mathcal D$.

    On the other hand, if $X_i'$ is $\mathcal T$, then it is filtered by
    quotients of elements of $\mathcal D$, so it admits a morphism from some
    element of $\mathcal D$, and since $X_i'$ is strictly above $X_i$,
 this element cannot be $X_i$.
    Thus $X_i$ cannot be replaced by $X_i'$ in this case
    either.

    Dual considerations explain why no element can be added to $\mathcal U$
    and no element of $\mathcal U$ can be lowered.
    \end{proof}

Putting together the previous results, we conclude the following theorem.
{
\renewcommand{\thetheorem}{\ref{thm:reptheory}}
\begin{theorem}
If $A$ is representation finite and $\mod A$ has no cycles, then
maximal orthogonal pairs correspond to torsion pairs, independent sets correspond to semibricks, and tight orthogonal pairs correspond to 2-simple-minded collections.
\end{theorem}
\addtocounter{theorem}{-1}
}

\section*{Acknowledgements}

We thank two diligent anonymous referees for helping us improve our exposition.  H.T. was partially supported by the Canada Research Chairs program and an NSERC Discovery Grant.  N.W. was partially supported by a Simons Foundation Collaboration Grant.

\end{document}